%% file: main_text.tex
\documentclass[]{article}

\input{preamble}

\crefname{cond}{condition}{conditions}
\creflabelformat{cond}{#2#1\@#3}
\crefname{obs}{observation}{observations}
\creflabelformat{obs}{#2#1\@#3}

\usepackage{commath}
\usepackage{authblk}
\usepackage{blindtext}
\usepackage{enumitem}
\usepackage{soul}
\usepackage[final]{pdfpages}
\usepackage{biblatex}
\title{CAT(0) Polygonal Complexes are 2-Median}
\author[1]{Shaked Bader}
\author[2]{Nir Lazarovich}
\affil[1]{Department of Mathematics, University of Oxford, Oxford, UK}
\affil[2]{Department of Mathematics, Technion—Israel Institute of Technology, Haifa, Israel}
\affil[ ]{\textbf{shaked.bader@gmail.com,\ nirlazarovich@gmail.com}}
\date{}

\begin{document}
\maketitle
\begin{abstract}
Median spaces are spaces in which for every three points the three intervals between them intersect at a single point. It is well known that rank-1 affine buildings are median spaces, but by a result of Haettel, higher rank buildings are not even coarse median.

We define the notion of ``2-median space'', which roughly says that for every four points the minimal discs filling the four geodesic triangles they span intersect in a point or a geodesic segment.
We show that CAT(0) Euclidean polygonal complexes, and in particular rank-2 affine buildings, are 2-median. 
In the appendix, we recover a special case of a result of Stadler of a Fary-Milnor type theorem and show in elementary tools that a minimal disc filling a geodesic triangle is injective.
\end{abstract}

\section{Introduction}
It is a well known fact that $\mathbb{R}$ and metric trees have the property that for any three points, the intersection of the geodesic segments between them is a single point. That is, metric trees are median in the following sense:

\begin{definition}
Let $(X,d)$ be a metric space 
\begin{itemize}
    \item For $x,y\in X$ we define the \emph{interval} between $x,y$ to be $$[x,y] = \set{z\;\mid\;d(x,z)+d(z,y) = d(x,y)}.$$
    
    \item The space $X$ is \emph{median} if for any $x,y,z\in X$, we have that $[x,y]\cap[x,z]\cap [y,z]=\set{m(x,y,z)}$ is a unique point.
\end{itemize}
\end{definition}

Examples of median spaces include: metric trees,  $\ell^1$ products of median spaces, $\ell^1(X)$ and $L^1(X,\mu)$, the vertices of a CAT(0) cube complexes with the induced graph metric, see Roller \cite{Roller}. In fact, by a result of Chepoi \cite{Chepoi} the class of 1-skeleta of CAT(0) cube complexes coincides with the class of median graphs.

Median spaces have been extensively studied: Niblo and Reeves \cite{niblo1997groups} proved that any action of a Property (T) group on a finite-dimensional CAT(0) cube complex has a global fixed point. 
In fact, Chatterji, Drutu and Haglund  \cite{Chatterji} prove that a locally compact second countable group has Property (T) if and only if any continuous action of it on a median space has bounded orbits.
Sapir \cite{sapir} proved that acting on a median space implies the Rapid Decay Property. 

Median algebras, which are spaces equipped with a map satisfying algebraic properties that the geometric median satisfies, are studied as well, see the work of Fiovaranti \cite{Fior1}, \cite{fior2} and Bader and Taller \cite{taller}.
In 2013 Bowditch defined a generalisation of a median space, called \emph{coarse median}, and showed that the mapping class group of a surface is a coarse median space. Niblo, Wright and Zhang \cite{Niblo} further simplified this definition and gave a definition of a \emph{coarse median algebra}.

In view of the above, we would like to have a higher-rank analogue. However, Haettel \cite{Haettel} proved that higher rank symmetric spaces and thick affine buildings which are not products of trees are not even coarse median spaces. 
This shows that one needs to modify the definition of median to fit the higher rank setting. We will focus on dimension two. 

As a first step, let us examine the case of the Euclidean plane $\bbE^2$. Clearly, $\bbE^2$ is not a median space as for a triple of non-collinear points the geodesic segments between them do not intersect. 
Instead, let us consider four points $x_1,x_2,x_3,x_4\in \mathbb{E}^2$ and the four full triangles they span $\set{\bt{x_i,x_j,x_k}}_{1\leq i<j<k\leq4}$. Note that if $x_1,x_2,x_3,x_4$ are non-collinear then the intersection of those triangles is a single point, as in \cref{fig:euc}, where in the figure we see the two different possible arrangements (the top and bottom of the figure) of the points in the plane and the four full triangles are coloured. We can see that $o$ is the unique point in the intersection of the four full triangles. We note that the Helly property for the intersection of convex sets gives us that the intersection exists in this case.

\begin{figure}[htp]
    \centering
    \includegraphics[]{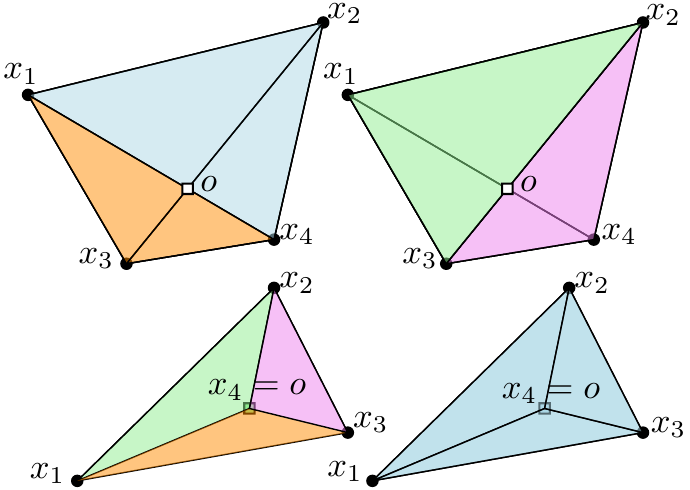}
    \caption{Four points in the Euclidean plane and the triangles they span}
    \label{fig:euc}
    \end{figure}
    
In the degenerate case, in which all four points are on a common line, we have that the intersection is an interval, it is exactly the interval $\overline{x_ix_j}\cap \overline{x_kx_l}$ for $\set{i,j,k,l}=\set{1,2,3,4}$ for which that intersection is non-empty.

We prove that a similar phenomenon, which we call \emph{2-median}, holds in 2-dimensional CAT(0) Euclidean polygonal complexes.
It is not clear to us what should ultimately be the general definition of a 2-median space, and we therefore refrain from giving such. 
Instead, we give an \emph{ad hoc} definition in our restricted setting of 2-dimensional CAT(0) Euclidean polygonal complexes, and show that it always holds.

Let $X$ be a 2-dimensional CAT(0) Euclidean polygonal complexes, and 
let $x,y,z\in X$. 
In \cref{triangle}, we define the full triangle $\bt{x,y,z}$ as the minimal subset of $X$ necessary in order to fill the geodesic triangle $\triangle(x,y,z)$ by a disc.
If $\triangle(x,y,z)$ is a Jordan curve, the full triangle $\bt{x,y,z}$ coincides with the image of the (unique) minimal disc in the sense of Lytchak and Wenger \cite{lytchak}. In this case, it was shown by Stadler \cite{stadler2021} that $\bt{x,y,z}$ is homeomorphic to a disc.

\begin{definition} \label{def of 2-median}
The space $X$ is a \emph{2-median space} if for any $x_1,x_2,x_3,x_4\in X$ either there exists $\set{i,j,k,l}=\set{1,2,3,4}$ such that $\overline{x_ix_j}\cap\overline{x_kx_l}$ is an interval, in which case $\bigcap_{1\leq i<j<k\leq 4}{\bt{x_i,x_j,x_k}}=\overline{x_ix_j}\cap\overline{x_kx_l}$, or $\bigcap_{1\leq i<j<k\leq 4}{\bt{x_i,x_j,x_k}}$ is a single point.
\end{definition}

Our main result is the following:

\begin{theoremA} \label{main theorem}
2-dimensional CAT(0) polygonal complexes are 2-median. 
\end{theoremA} 

In \cite{KLa}, Kleiner and Lang define and prove higher rank hyperbolicity properties in a general setting. In particular, they prove that 2-dimensional CAT(0) polygonal complexes have slim simplices; that is, if $f$ maps the boundary of the 3-simplex to a 2-dimensional CAT(0) polygonal complex $X$, such that $f$ is a $(L,a)$ quasi isometry on each facet, then the image of each facet is contained in the $D=D(X,L,a)$ neighbourhood of the image of the other facets.

As we said, we refrain from giving a general definition of 2-median spaces. However, speculating that such a definition will resemble \cref{def of 2-median} (after the appropriate notions of intervals and full triangles are defined), we expect the following:
the 2-median property should essentially be a rank-two phenomenon -- e.g. the intersection of the triangles spanned by four generic points in the 3-dimensional Euclidean space $\bbE^3$ is empty; it should be a non-positive curvature phenomenon -- e.g. the intersection of the triangles of the regular tetrahedron in the sphere $\bbS^2$ is empty;
full triangles should be minimal discs and not convex hulls -- cf. \cref{convex hulls example}. 

The outline of the paper is as follows:

\S\ref{sec: prelims} and \S\ref{sec: planar complexes} are introductory: in \S\ref{sec: prelims} we review some properties of CAT(0) Euclidean polygonal complexes and their links and in \S\ref{sec: planar complexes} we define a generalized disc and a generalized annulus which we will use throughout the text.

In \S\ref{sec: triangles in CAT0} we give an ad hoc topological definition of $\bt{x,y,z}$ for a CAT(0) polygonal complex and characterize it. 

In \S\ref{sec: minimal triangle} we prove using a special case of a result by Stadler \cite{stadler2021}, that in this setting our ad hoc definition coincides with the definition above.

In \S\ref{sec: existence} we show that the intersection of the four full triangles is non-empty by showing that $\bt{x_1,x_2,x_3}\subseteq \bigcup_{(i,j,k)\neq(1,2,3)}\bt{x_i,x_j,x_k}$, and using  Sperner's Lemma.

In \S\ref{sec: discreteness} we discuss the discreteness of the intersection. We show that the intersection of the four full triangles is discrete while the intersection of two of the triangles does not have isolated points.

In \S\ref{sec: 1d case} we prove the theorem in the degenerate case.

In \S\ref{sec: near-immersions} we prove that a \emph{near-immersion} from a disc filling an embedded polygonal curve is injective:

\begin{definition} \label{near-immersion}[See \cite{Fansandladders} Definition 2.13]
    A cellular map $Y\to X$ between two CW-complexes is a \emph{near-immersion} if the restriction $Y-Y^{(0)}\to X$ is locally injective.
\end{definition}

\begin{theorem} \label{immersion}
Let $X$ be a CAT(0) polygonal complex, and let $\Gamma$ be an embedded polygonal curve with ordered vertices $x_1,...,x_n$ such that $\kappa(\Gamma)<4\pi$. Assume $f:\Delta \to X$ is a near-immersion, where $\Delta$ is a finite $\Delta$-complex homeomorphic to a disc and $f|_{\partial \Delta} : \partial \Delta \to \Gamma$ is a homeomorphism, then $f$ is injective.
\end{theorem} 

In \S\ref{sec: minimal tetrahedron} we prove the theorem in the general case. The idea is to obtain a near-immersion from a \emph{deflated tetrahedron} (see \cref{fig:deflated tetra})
to our space and show that it is actually injective. We obtain a near-immersion by taking any cellular map from a deflated tetrahedron to our space and folding, and showing that when the procedure terminates we remain with a deflated tetrahedron.

\begin{figure}[htp]
    \centering
    \includegraphics[]{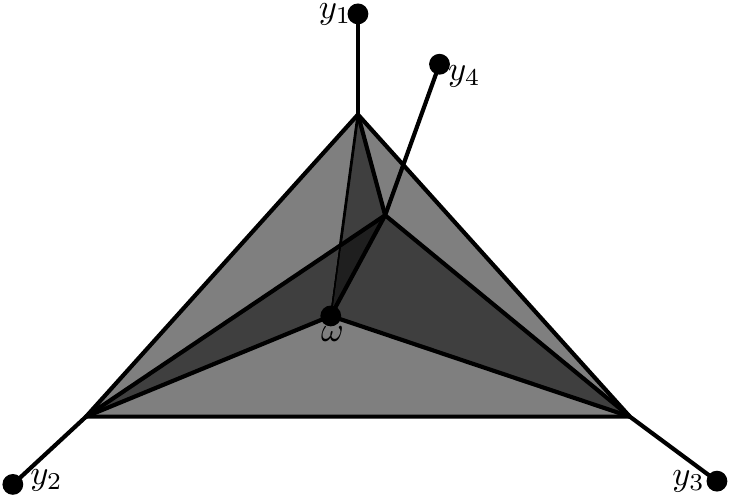}
    \caption{Deflated tetrahedron}
    \label{fig:deflated tetra}
    \end{figure}

      Injectivity will follow using \cref{immersion} and the two facts we prove in \S\ref{sec: discreteness} about intersections of two and four full triangles.

In \S\ref{sec:conclusion} we discuss further directions and in \S\ref{sec: appendix} we give an elementary proof of Stadler's result for CAT(0) polygonal complexes.

\paragraph{Acknowledgements.} 
The authors would like to thank Tobias Hartnick and Nima Hoda for fruitful discussions at an early stage of this project, and Michah Sageev for numerous helpful suggestions and corrections of this manuscript.
The authors are indebted to the anonymous referee for the careful reading of this manuscript and the valuable suggested improvements.
Both authors were partially supported by the Israel Science Foundation (grant no. 1562/19).








\section{CAT(0) polygonal complexes and their links}
\label{sec: prelims}
We assume that the reader is familiar with CAT(0) spaces. Specifically, we will focus on \emph{Euclidean CAT(0) polygonal complexes} i.e. CAT(0) polygonal complexes whose cells are convex Euclidean polygons (not assumed to be regular). The relevant definitions can be found in parts I and II of \cite{BridHäf}.

\begin{notation}

\begin{itemize}
    \item 
If $\xi,\zeta$ are points in $\lk(p)$, we denote by $\angle_p(\xi,\zeta)$ the angle between $\xi$ and $\zeta$ in $X$ (as defined in \cite[Definition I.1.12]{BridHäf}). 
This defines a metric $\angle_p$ on $\lk(x)$.
We denote the length of a path $\alpha$ in $\lk(p)$ by $\measuredangle_p \alpha$, and by $\measuredangle_p(\xi,\zeta)$ the induced length metric, i.e. the length of the shortest path connecting them in $\lk(p)$. Note that $\angle_p$ does not necessarily equal the metric $\measuredangle_p$. 

We will abuse notation and for $x,y\in X$ and $\gamma$ a path in $X$ we denote by $\angle_p(x,y),\measuredangle_p(\alpha),\measuredangle_p(x,y)$ the angle, length and distance of the projections of $x,y$ and $\gamma$ to $\lk(p)$. 

In cases in which the point referred to is clear, we will sometimes omit the subscript $p$, 

\item Let $Y$ be a geodesic space, $x,y\in Y$. When there exists a unique geodesic between $x,y$ (as is the case in CAT(0) spaces and when $x,y\in \lk(q)$ and $\angle_q(x,y)<\pi$ for $q$ in a CAT(0) space), we denote it by $\overline{xy}$.
\end{itemize}

\end{notation}

The CAT(0) condition for CAT(0) polygonal complexes takes a simple form:

\begin{claim}[{Link condition \cite[Theorem II.5.2]{BridHäf}}] \label{link condition}
A Euclidean polygonal complex is CAT(0) if and only if it is simply connected and for each 0-cell $v$ we have that the length of every injective loop in $\lk(v)$ (i.e. the \emph{girth} of $\lk(v)$) has length at least $2\pi$.
\end{claim}

In the remainder of this section, we discuss the connection between the link of point $p\in X$ and $X-p$.

Recall that there exists $\epsilon>0$ such that the sphere around $p$ of radius $\epsilon$, $S$, can be identified with $\lk(p)$, and the closest point projection, $\pi_p:X-p\to S$, is a deformation retraction of $X-p$ (see \cite[Proposition II.2.4]{BridHäf}).

When it is clear which point we are referring to, we will denote $\pi_p(x)$ by $\hat{x}$.

\begin{proposition} \label{path in link}
Let $X$ be a CAT(0) polygonal complex, $p\in X$.
 Let $x,y\in X$, then $\angle_p(x,y)=\min\set{\measuredangle_p(\hat{x},\hat{y}),\pi}$.

    In particular, if $p\notin \overline{xy}$, then $\angle_p(x,y)<\pi$, so $\angle_p(x,y)=\measuredangle_p(\hat{x},\hat{y})$ and there exists a unique path in $\lk(p)$ between $\hat{x},\hat{y}$ of length $\angle_p (x,y)$.

\end{proposition} 

\begin{proof}
By \cite[Theorem I.7.16]{BridHäf} there exists $r>0$ such that $S$, the sphere or radius $r$ around $p$, is homeomorphic to $\lk(p)$ and the ball or radius $r$ around $p$ is isometric to $C(\lk(p))$; where $C(\lk(p))$ is the cone over $\lk(p)$ with the metric defined such that for $u,v\in \lk(p)$ the angle between $\overline{pu},\overline{pv}$ in $\lk(p)$ is $\min\set{\measuredangle_p(x,y),\pi}$.

We can assume $x,y\in S$.
As isometries preserve angles, we have that $\angle_p(x,y)=\min\set{\measuredangle_p(x,y),\pi}$. Hence, if $\angle_p(x,y)<\pi$, it must be that $\measuredangle_p(x,y)=\angle_p(x,y)$.
\end{proof}

As the next example shows, in general if $X$ is a CAT(0) polygonal complex, $p\in X$ and $q,r\in X$ such that $p\notin \overline{qr}$, it might be that $\pi_p(\overline{qr})\neq \overline{\pi_p(q)\pi_p(r)}$. However, we will see in \cref{projection to link} that $\pi_p(\overline{qr}),\overline{\pi_p(q)\pi_p(r)}$ are homotopic relative to endpoints in $\lk(p)$.

\begin{example}
Let $T$ be the metric graph with one vertex $o$ of degree three and three 1-cells of length 1 with vertices $a,b,c$.
Consider $T\times [0,1]$ with its CAT(0) metric, that is the $\ell^2$ metric, and the points $p=(a,1),q=(b,0),r=(c,0)$, as in \cref{fig:proj not geod}. Let $s$ be the intersection of $\overline{qr}$ with $\set{o}\times [0,1]$, that is, $s=(o,0)$. Note that $\pi_p(s)\notin \overline{\pi_p(q)\pi_p(r)}=\set{\pi_p(q)}$.
\end{example}

\begin{figure}[H]
    \centering
    \includegraphics[]{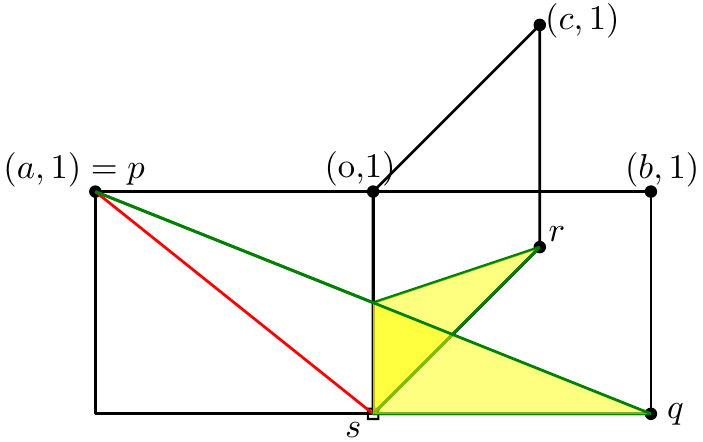}
    \caption{Example in which $\pi_p(\overline{qr})\neq \overline{\pi_p(q)\pi_p(r)}$}
    \label{fig:proj not geod}
    \end{figure}

\begin{claim} \label{projection to link}
Let $X$ be a CAT(0) polygonal complex and $p\in X$. 
Given $q,r\in X$ such that $p\notin \overline{qr}$ we have that $\pi_p(\overline{qr})$ is homotopic relative to endpoints in $\lk(p)$ to the geodesic between $\pi_p(q),\pi_p(r)$. In particular, there exists a unique geodesic between $\pi_p(q),\pi_p(r)$.
\end{claim}

\begin{proof}
Let $S$ be as in \cref{path in link} and identify $\lk(p)$ with $S$.
Let $\alpha_q:[0,1]\to X$ be a parametrisation of $\overline{q\hat{q}}$ where $\hat{q}\in S=\lk(p)$, such that $\alpha_q(0)=q$ and $\alpha_q(1)=\hat{q}$. Similarly consider $\alpha_r$. Let $\gamma_t:[0,1]\to X$ be a parametrisation of $\overline{\alpha_q(t)\alpha_r(t)}$. Let $G:[0,1]\times[0,1]\to X-p$ be $G(s,t)=\gamma_t(s)$. 

We have that $G$ is continuous and well defined into $X-p$, so $\pi_p\circ G$ gives a homotopy between $\pi_p\circ \gamma_0$ and $\pi_p\circ\gamma_1$.

As in \cref{path in link}, the ball bounded by $S$ is isometric to the cone over $S$. The cone over $\overline{\hat{q}\hat{r}}$ is convex and contains $\gamma_1$, so $\pi_p\circ \gamma_1=\overline{\hat{q}\hat{r}}$. We get that $\pi_p\circ G$ is a homotopy relative to $\set{\hat{q},\hat{r}}$ between $\pi_p(\overline{qr})$ and $\overline{\hat{q}\hat{r}}=\overline{\pi_p(q)\pi_p(r)}$.

\end{proof}

As an immediate corollary, we get the following. 

\begin{corollary} \label{homotopies in link}
Let $X$ be a CAT(0) polygonal complex, let $\Delta$ be a geodesic triangle with vertices $x,y,z\in X$, and let $p\in X$. 

\begin{itemize}
    \item If $p\in X-\Delta$, then $\pi_p(\Delta)$ is homotopic to the geodesic triangle in $\lk(p)$ with vertices $\hat{x},\hat{y},\hat{z}$.
    
    \item If $p\in \overline{xy}-(\overline{xz}\cup\overline{yz})$ then $\pi_p(\Delta-p)$ is homotopic to $\overline{\hat{x}\hat{z}}\cup\overline{\hat{y}\hat{z}}$.
    
    \item If $p\in\overline{xy}\cap\overline{xz} - \overline{yz}$ then $\pi_p(\Delta-p)$ is homotopic to $\set{\hat{x}}\cup \overline{\hat{y}\hat{z}}$ if $p\neq x$ and to $\overline{\hat{y}\hat{z}}$ if $p=x$.
    
    \item If $p\in \overline{xy}\cap \overline{xz}\cap \overline{yz}$ then $\pi_p(\Delta-p)=\set{\hat{u}\mid u\in\set{x,y,z}-p}$.
\end{itemize}
where all homotopies are relative to $\set{\hat{x},\hat{y},\hat{z}}$ when they are defined.
See \cref{fig:link of p} for an illustration of the four cases.
\end{corollary}

\begin{figure}[H]
    \centering
    \includegraphics[]{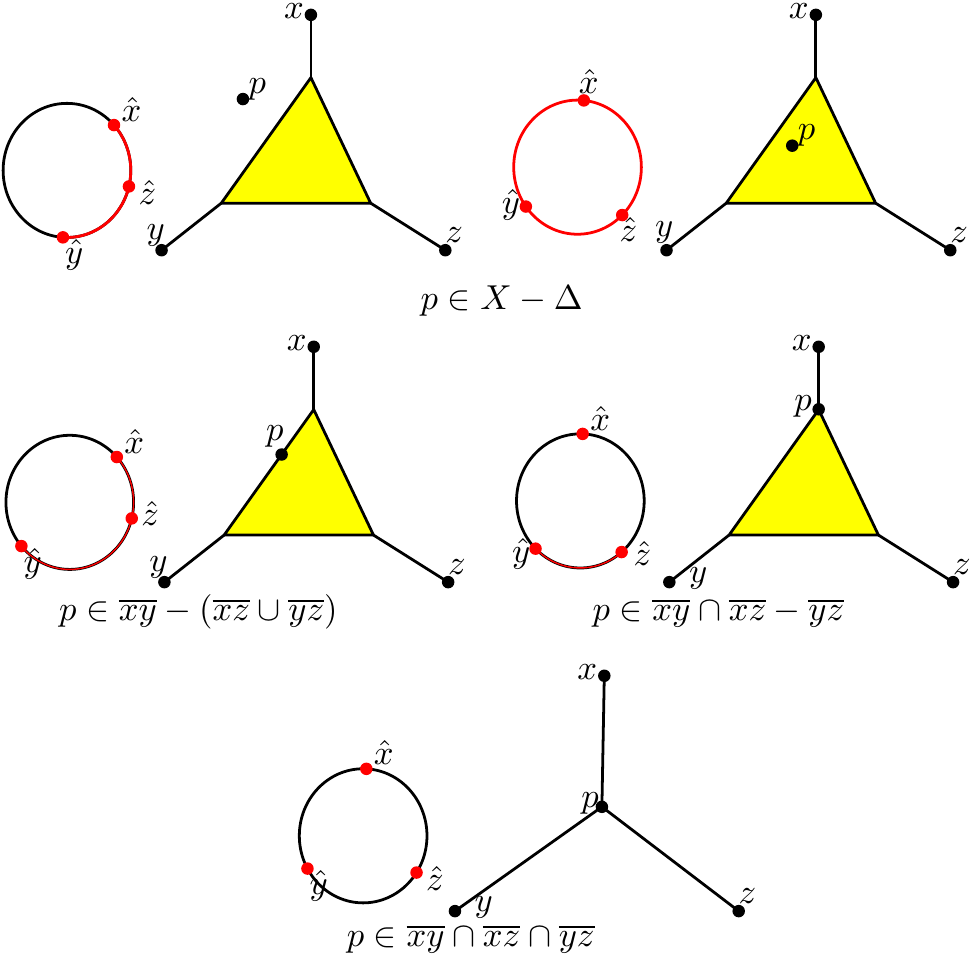}
    \caption{The four cases, $\lk(p)$ on the left of the configuration}
    \label{fig:link of p}
    \end{figure}

\section{Some planar complexes}
\label{sec: planar complexes}

In what follows we will need to discuss complexes which are similar to discs and annuli. We will record in this section the technicalities of the definitions.

\begin{definition} \label{generalized disc and triangle}
\begin{itemize}
    \item A \emph{generalized disc} $D$ is a contractible finite $\Delta$-complex (see \cite{Hatcher} Chapter 2 for definition of $\Delta$-complex) with a piecewise linear embedding to $\bbR^2$, see \cref{fig:generalized disc}. A \emph{boundary cycle} is a
    map $\alpha: C \to D^{(1)}$ such that $\alpha(C)$ is contained in the topological boundary, contains all 1-cells that are contained in the topological boundary and does not cross itself, that is, if $e_i,e_{i+1}$ are consecutive 1-cells in $\alpha$ then $e_i^{-1},e_{i+1}$ are consecutive in the cyclic order coming from the planar embedding of all 1-cells emanating from $t(e_i)$, and where $C$ is cycle graph which is of minimal length among all such graphs for which there exists such a map. The boundary cycle can be thought of as the boundary of an epsilon-neighbourhood of the embedded generalized disk.
    
    As complexes, they are just a connected and simply connected union of disks and trees. But it is important to note that generalized disks come with their embedding into $\mathbb{R}^2$.
    
    Any two parametrized boundary cycles of $D$ differ by precomposition by an automorphism of $C$. By abuse of language, we call $\alpha$ \emph{the} boundary cycle of $D$.
    
    \item A \emph{generalized triangle} is a generalized disc $D$ with boundary cycle $\alpha: C\to D$ together with 3 vertices $x,y,z\in C$ on its boundary. 
    We denote $\overline{xy}^{D} = \alpha |_{\overline{xy}}$ where $\overline{xy}$ is the path in $C$ connecting $x$ and $y$ and not passing through $z$.
    The images of $x,y,z$ and the arcs $\overline{xy},\overline{xz},\overline{yz}$ under $\alpha$ are well-defined and by abuse of notation we regard the points $x,y,z$ as the corresponding points in $D$.
    
    \item Given a generalized disc $D$, points in $D$ that are the image of two points in the boundary cycle of $D$ are referred to as \emph{skinny parts of $D$}. 
\end{itemize}
\end{definition}

\begin{figure}[H]
    \centering
    \includegraphics[]{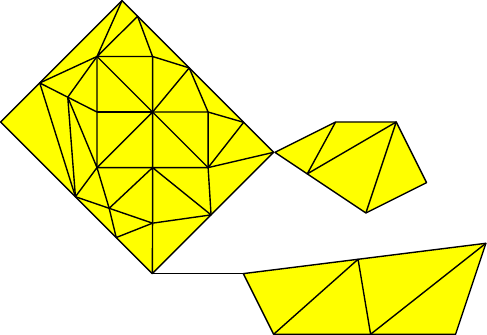}
    \caption{Example of generalized disc}
    \label{fig:generalized disc}
    \end{figure}

\begin{definition} \label{generalized annulus}
A \emph{generalized annulus} is a finite $\Delta$-complex with a piecewise linear embedding to $\mathbb{R}^2$ which is homotopy equivalent to $\mathbb{S}^1$ and has no separating points. For example see \cref{fig:generalized annulus}.

Given a generalized annulus $A$, it separates $\mathbb{R}^2$ into a bounded and an unbounded component where the bounded component is a generalized disc $D$. Let the \emph{inner boundary} of $A$ be the boundary cycle of $D$ and similarly $A\cup D$ is a generalized disc and we define the \emph{outer boundary} of $A$ to be the boundary cycle of $A\cup D$.
\end{definition}

Note that the inner boundary of a generalized annulus is a circle, as it has no separating points.

\begin{definition} \label{good gluing}
Let $A$ be a generalized annulus and let $\alpha:C\to A^{(1)}$ be the inner boundary of $A$ where $C$ is a cycle graph.
The equivalence relation $\sim$ on $\alpha(C)$ is a \emph{good gluing} if the relation $\sim$ identifies $\alpha(e_{i})$ with $\alpha(j_i)$ where $e_1,\dots,e_k,j_k^{-1},...,j_1^{-1}$ are the edges of $C$ in this order, and at most one of $\inner(e_i),\inner(j_i)$ is on the outer boundary of $A$. For example see \cref{fig:generalized annulus}.
\end{definition}

\begin{figure}[H]
    \centering
    \includegraphics[]{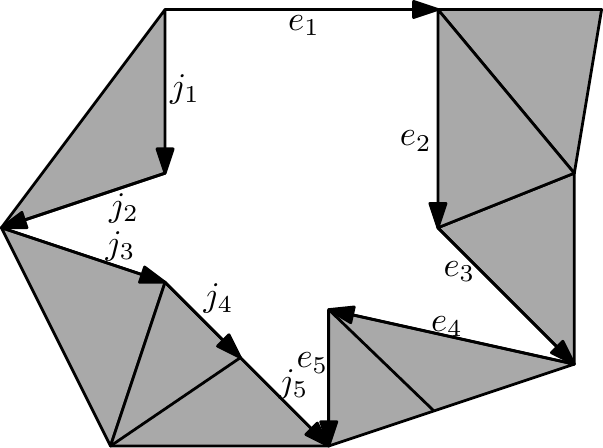}
    \caption{generalized annulus with good gluing}
    \label{fig:generalized annulus}
    \end{figure}

\begin{claim} \label{gluing an annulus}
If $A$ is a generalized annulus and $\sim$ is a good gluing of $A$, then $A/\sim$ is homeomorphic to a disc.
\end{claim}

\begin{proof}
Assume $A$ is a generalized annulus and $\sim$ is a good gluing of $A$. The quotient $A/\sim$ is a surface: indeed, consider the links of vertices in $A$ - they are subsets of the circle that have at most two components (as otherwise either the vertex would be a cut vertex or $A$ would not be homotopy equivalent to $\bbS^1$), the good gluing ensures that they glue up to have one component which is either a circle or a non-trivial arc. 
In fact, we see that the image of the outer boundary of $A$ is the only boundary component of the surface $A/\sim$.

The generalized annulus $A$ is homotopic to a circle and so $\chi(A)=0$. 
The CW complex $A/\sim$ has the same number of 2-cells as $A$, $k$ fewer 1-cells, and $k-1$ fewer 0-cells where $k$ is as in \cref{good gluing}. It follows that $\chi(A/\sim)=\chi(A)+1=1$.
Since the compact surface $A/\sim$ has one boundary component and $\chi(A/\sim)=1$ it must be a disc.
\end{proof}

\section{Triangles in CAT(0) polygonal complexes}
\label{sec: triangles in CAT0}
Throughout let $X$ be a CAT(0) polygonal complex. E.g. a rank two affine building. Note that a polygonal complex can be subdivided to be a triangle complex, so without loss of generality, we assume that $X$ is a CAT(0) triangle complex, i.e. each 2-cell is a Euclidean triangle, which is not assumed to be regular.

\begin{definition} \label{triangle}
Let $x,y,z\in X$.
\begin{enumerate}
    \item Denote by $\triangle(x,y,z)=\overline{xy}\cup\overline{yz}\cup \overline{zx}$ the \emph{geodesic triangle} spanned by $x,y,z$.
    \item Define the \emph{full triangle} $\bt{x,y,z}$ to be the union of $\triangle(x,y,z)$ and the set of points $p$ in $X-\triangle(x,y,z)$ for which $\triangle(x,y,z)$ is not trivial in $\pi_1(X-p)$. 
\end{enumerate}
\end{definition}

\begin{example}\label{convex hulls example}
In \cref{fig:proj not geod} we see $\triangle(p,q,r)$ in green in the space $T\times[0,1]$ and its corresponding full triangle in yellow.
Note that $\bt{p,q,r}$ is not convex - the geodesic between $s$ and $p$ does not lie inside of it. 

Our uniqueness result would be false if we had considered the convex hull of the three points to be the full triangle, as one could see from \cref{fig: non-uniqueness convex hull}: the purple line is not in the intersection of the four triangles, but is in the intersection of the four convex hulls, as it is not in $\bt{x,y,w}$, but is in $\conv(x,y,w)$, and it is in $\bt{x,z,w}=\conv(x,z,w)$ and in $\overline{zy}\subseteq \bt{x,y,z}\cap\bt{y,z,w}$. So the intersection of the four convex hulls is not a single point.

\end{example}

\begin{figure}[htp]
    \centering
    \includegraphics[]{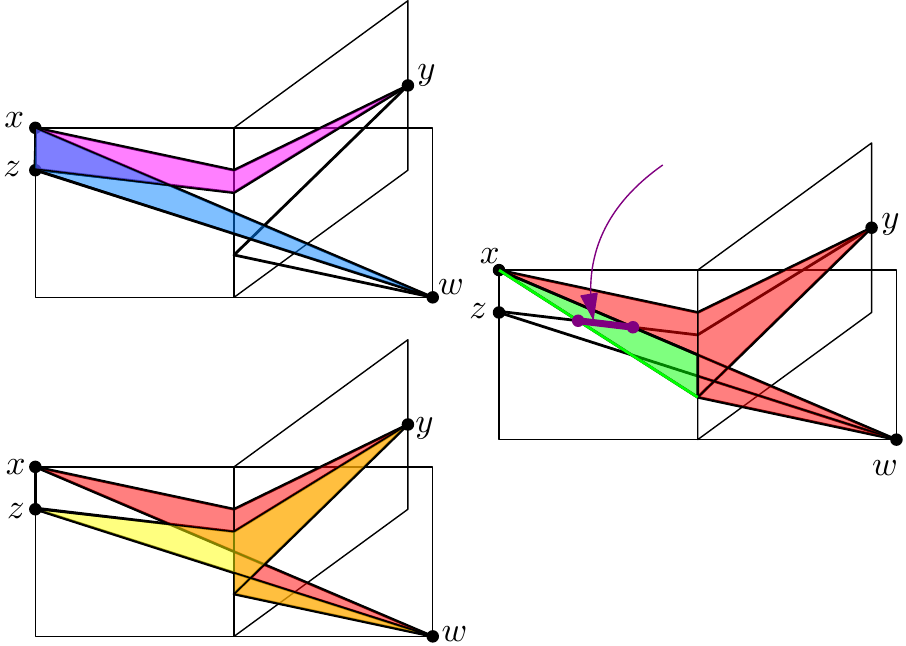}
    \caption{Intersection of the convex hulls is not a single point. Top left: $\bt{x,y,z}=\conv(x,y,z)$ and $\bt{x,z,w}=\conv(x,z,w)$; bottom left: $\bt{x,y,w},\bt{z,y,w}$; right: $\conv(x,y,w)$}
    \label{fig: non-uniqueness convex hull}
\end{figure}

In \cref{fig:proj not geod} the full triangle is not a disc, but rather a disc with a ``spike'' attached to it. This leads us to the following definition.

\begin{definition}
A \emph{spiky triangle (resp. full spiky triangle)} is the topological space obtained from gluing a possibly degenerate interval to each of the vertices of a possibly degenerate Euclidean triangle (resp. full Euclidean triangle), as in \cref{fig:spiky triangle}.

\begin{figure}[htp]
    \centering
    \includegraphics[]{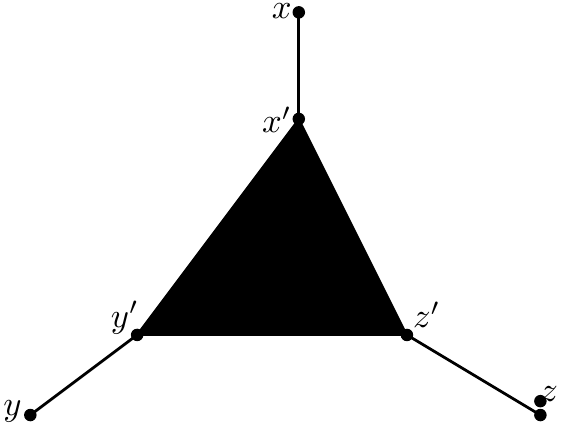}
    \caption{A full spiky triangle}
    \label{fig:spiky triangle}
\end{figure}
\end{definition} 

\begin{lemma}\label{geodesic triangle is spiky}
Let $x,y,z\in X$. Then the geodesic triangle $\triangle(x,y,z)$ is a spiky triangle.
\end{lemma}

\begin{proof}
The intersection of any two geodesics in a CAT(0) space is either empty or a (possibly degenerate) segment. Consider $\overline{xy}\cap \overline{xz}$, this is a (possibly degenerate) segment, as it is non-empty. Let $x'$ be the point furthest away from $x$ on that interval and 
define $y',z'$ similarly. We get $\triangle(x,y,z)=\triangle(x',y',z')\cup \overline{xx'}\cup\overline{yy'}\cup \overline{zz'}$, as in \cref{fig:spiky triangle}.

It is easy to see that unless $\triangle(x,y,z)$ is a tripod,  $\overline{xx'}\cap \overline{yz}=\emptyset, \overline{yy'}\cap \overline{xz}=\emptyset, \overline{zz'}\cap \overline{xy}=\emptyset$. It follows that $\triangle(x,y,z)$ is a spiky triangle.
\end{proof}

A similar phenomenon occurs in the link of a point in $X$. That is

\begin{lemma} \label{triangles are spiky in link}
Let $G$ be a metric graph with girth at least $2\pi$ and let $x,y,z\in X$ be such that $d(x,y),d(y,z),d(x,z)<\pi$, then the geodesic triangle $\triangle(x,y,z)$ (which is well defined by assumption on the distances between the points and the girth) is a spiky triangle.
\end{lemma}

The proof of this lemma is the same as the proof of \cref{geodesic triangle is spiky} where the sentence ``The intersection of any two geodesics in a CAT(0) space is either empty or a (possibly degenerate) segment'' is replaced with ``The intersection of any two geodesics of length $<\pi$ in $G$ is either empty or a (possibly degenerate) segment by assumption on the girth''.

We want to understand how $\bt{x,y,z}$ looks locally.
Consider $x,y,z\in X$ and $q\in \bt{x,y,z}$.
Let $B$ be a small enough ball around $q$. We want to describe $B\cap \bt{x,y,z}$.

We start by analyzing points $q$ in $\bt{x,y,z} - \triangle(x,y,z)$.
By \cref{path in link}, $q\in \bt{x,y,z}-\triangle(x,y,z)$ if and only if  $c=\pi_q(\triangle(x,y,z))$ is a homotopically non-trivial cycle in $\lk(q)$ or $q\in \triangle(x,y,z)$.
We claim that for a small enough ball $B$ around $q$ the points of $B\cap \bt{x,y,z}$ are the cone over the ``essential part'' $C$ of $c$ which is defined below:

\begin{definition}
Let $G$ be a graph. Given a cycle $c$ in $G$ which is not trivial in homotopy (resp. a path), the \emph{essential part} of $c$ is the unique closed non-backtracking cycle (resp. path) which is homotopic (resp. homotopic relative to end points) to $c$.
Given $c$ a disjoint union of cycles and paths in $G$, the \emph{essential part} of $c$ is the union of the essential parts of the connected components of $c$.
\end{definition}

\begin{claim} \label{triangle in link 1}
Let $q\in \bt{x,y,z}-\triangle(x,y,z)$ and let $B$ be a small enough ball around $q$ such that $\partial B\cong \lk(q)$ and $B\cap \triangle(x,y,z)=\emptyset$. Let $C$ be the essential part of $\pi_q(\triangle(x,y,z))$. Then for $p\in B$, $\hat{p}\in C$ if and only if $p\in \bt{x,y,z}$.
\end{claim}

\begin{proof}
By \cref{homotopies in link}, $\triangle(x,y,z)$ is homotopic to $\triangle(\hat{x},\hat{y},\hat{z})\subseteq\lk(q)\cong \partial B$ and $\measuredangle(\hat{x},\hat{y}),\measuredangle(\hat{x},\hat{z}),\measuredangle(\hat{y},\hat{z})<\pi$ by assumption that $q\notin\triangle(x,y,z)$. By \cref{triangles are spiky in link}, $\triangle(\hat{x},\hat{y},\hat{z})$ is a spiky triangle, and since $q\in \bt{x,y,z}$, it is non-degenerate. The essential part of $\triangle(\hat{x},\hat{y},\hat{z})$, $C$, is thus an embedded circle in $\partial B$. 
In conclusion, we got a homotopy in $X-p$ of $\triangle(x,y,z)\to C$.

Let $\hat{C}$ be the cone over $C$ with respect to $q$. It is a convex disc and $\hat{p}\in C$ if and only if $p$ is in the cone. 
If $p \notin \hat{C}$ then $C$ is filled with a disk $\hat{C}$ in $X-p$, and so $p\notin \bt{x,y,z}$. If $p \in \hat{C}$ then $C$ projects to a non-trivial loop in $\lk(p)$ (since $\hat{C}$ is convex). It follows that $C$ is non trivial in $\pi_1(X-p)$, whence $p\in\bt{x,y,z}$.

\end{proof}

Similarly one proves the following claim

\begin{claim} \label{triangle in link 2}
Let $q\in \triangle(x,y,z)$ and let $B$ be a small enough ball around $q$ such that $\partial B\cong \lk(q)$ and $B$ does not intersect any side of $\triangle(x,y,z)$ not containing $q$. Let $C$ be the essential part of $\pi_q(\triangle(x,y,z)-q)$. Then for $p\in B$, $\hat{p}\in C$ if and only if $p\in \bt{x,y,z}$.
\end{claim}

\begin{proof}
    We argue similarly to the proof of \cref{triangle in link 1}, by analysing $\hat{C}$ (in the notations of the proof of \cref{triangle in link 1}) for the different cases in \cref{homotopies in link}.
\end{proof}

We remind that \cref{homotopies in link} characterizes the connected components of $\pi_q(\triangle(x,y,z)-q)$.

In \cref{triangles are discs} we will show that full triangles in $X$ are homeomorphic to full spiky triangles.

\begin{definition}
We say $x_1,x_2,x_3\in X$ are \emph{strongly non-collinear} if they are distinct and for any $\set{i,j,k}=\set{1,2,3}$ we have $\overline{x_ix_j}\cap \overline{x_ix_k}=\set{x_i}$.
\end{definition}

\begin{definition}
The points $x_1,x_2,x_3,x_4\in X$ form a \emph{lanky quadrilateral} if $\overline{x_ix_j}\cap\overline{x_kx_l}$ is a non-degenerate segment for some $\set{i,j,k,l}=\set{1,2,3,4}$.
E.g. see \cref{fig:1dim}. Note that this does not imply that the four points lie on a geodesic.
\end{definition}

\begin{figure}[H]
    \centering
    \includegraphics[]{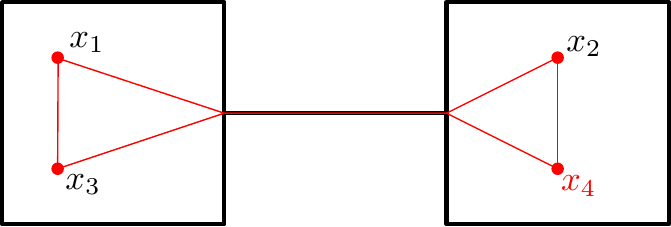}
    \caption{$x_1,x_2,x_3,x_4$ form a lanky quadrilateral}
    \label{fig:1dim}
    \end{figure}

In the definition of 2-median, \cref{def of 2-median}, there are two cases to consider: either the four points form a lanky quadrilateral, in which case the intersection of the four full triangles is an interval, or not and then the intersection is a point.

\section{Minimal triangles}
\label{sec: minimal triangle}
In \cref{geodesic triangle is spiky} we showed that geodesic triangles in $X$ are homeomorphic to spiky triangles. In this section, we prove that full triangles in $X$ are homeomorphic to full spiky triangles.

\begin{theorem} \label{triangles are discs} Let $x,y,z\in X$ and let $T$ be an (abstract) full spiky triangle with $\partial T\cong \triangle(x,y,z)$. Then there exists an embedding $f:T\hookrightarrow X$ such that $f|_{\partial T}$ maps $\partial T$ homeomorphically onto $\triangle(x,y,z)$. 

Any such map has image $\bt{x,y,z}$.
\end{theorem}

We will prove the theorem first in the strongly non-collinear case and then deduce it to the general case. 
The existence of an embedding in the strongly non-collinear case is a special case of Theorem 78 in \cite{stadler2021}. We give a more elementary proof in our case in the appendix, see \cref{Fary-Milnor}.

The uniqueness of the image in the strongly non-collinear case follows from the following lemma.

\begin{lemma}\label{injective discs}
Assume $\Gamma$ is a piecewise linear curve in a CAT(0) polygonal complex $X$ and assume there exists a $\Delta$-complex homeomorphic to $\bbD^2$ and a cellular and injective map $f:D\to X$ such that $f(\partial D)=\Gamma$, then $f(D)$ equals the union of $\Gamma$ with the set of points $p\in X-\Gamma$ such that $\Gamma$ is not trivial in $\pi_1(X-p)$.

In particular, for any $g:\mathbb{D}^2\to X$ cellular with respect to some $\Delta$-complex structure on $\bbD^2$ and injective such that $g(\partial \bbD)=\Gamma$, we have $g(\mathbb{D}^2)=f(D)$.
\end{lemma}

\begin{proof}
Denote by $M$ the union of $\Gamma$ with the set of points $p\in X$ such that $\Gamma$ is not trivial in $\pi_1(X-p)$. Clearly $M\subseteq f(D)$, as if $p\notin f(D)$ then $f$ provides a null homotopy of $\Gamma$ in $X-p$ and thus $p\notin M$.

For the other direction, let $p\in D$. There are projections, as in \cref{projection to link}, $\pi_p:\partial D\to \lk(p)$ and $\pi_{f(p)}:X-f(p)\to \lk(f(p))$ and as $f$ is injective, we have $f_p\circ \pi_p$ is homotopic equivalent to $\pi_{f(p)}\circ f$.

As $\pi_p(\partial D)$ is a cycle, so is $f_p(\pi_p(\partial D))$, hence it is not null homotopic in $\lk(f(p))$ which is a graph. By the equality above $\pi_{f(p)}\circ f(\partial D)=\pi_{f(p)}(\Gamma)$ is not trivial, so $\Gamma$ is not trivial in $\pi_1(X-p)$.
\end{proof}

Now for the general case.

\begin{proof}[proof of \cref{triangles are discs}]
The existence of such a map follows from the strongly non-collinear case using the following lemma.

\begin{lemma}\label{triangle inj implies spiky triangle}
    If $f$ is a map from an abstract spiky triangle $T$ to $X$ such that it maps $\partial T$ homeomorphically onto $\triangle(x,y,z)$ and the triangle part of $T$ is mapped injectively, then $f$ is injective.
\end{lemma}

\begin{proof}

Let $x,y,z\in X$ and denote by $x'\in X$ the point in which $\overline{xy},\overline{xz}$ diverge and similarly denote $y',z'$, as in \cref{fig:spiky triangle}.
By \cref{geodesic triangle is spiky}, $\triangle(x,y,z)$ is homeomorphic to a spiky triangle and $\triangle(x',y',z')$ is homeomorphic to a triangle.

By Theorem 78 in \cite{stadler2021}, there exists a full triangle $D$ and $f:D\to X$ an embedding such that $f|_{\partial D}:\partial D\to \triangle(x',y',z')$ is a homeomorphism.

By \cref{injective discs} $f(D)=\bt{x',y',z'}$.

Glue to $D$ the three spikes, i.e. let $I_1,I_2,I_3$ be three intervals homeomorphic to $\overline{xx'}, \overline{yy'}, \overline{zz'}$ respectively and let $a_j\in I_j$ be one of their endpoints. 
Consider $T=D\sqcup I_1\sqcup I_2\sqcup I_3/\sim$ where $a_1\sim f^{-1}(x'),a_2\sim f^{-1}(y'),a_3\sim f^{-1}(z')$ and consider the map $\bar{f}:T\to X$, induced from $f$ and the homeomorphisms of $I_j$ to the geodesic segments.

In order to show $\bar{f}$ is injective it is enough to show that for $p\in \partial T-D$ we have $\bar{f}(p)\notin \bar{f}(D)=\bt{x',y',z'}$. Indeed as $\bar{f}|_{\partial T}$ is a homeomoprhism, we have that $\bar{f}(p)\notin \triangle(x',y',z')$, hence, by \cref{triangle in link 1}, $\bar{f}(p)\in \bt{x',y',z'}$ if and only if the concatenation of $\overline{\hat{x'}\hat{y'}}, \overline{\hat{y'}\hat{z'}}, \overline{\hat{z'}\hat{x'}}$ in $\lk(p)$ is not null homotopic. But two out of the three segments are constant in $\lk(p)$, as they are continuations of geodesics emanating from $p$, so the union of the three is null homotopic.
\end{proof}

In conclusion, we found an injective map from an abstract full spiky triangle to $X$ mapping the boundary of the triangle homeomorphically onto $\triangle(x,y,z)$. We are left to prove that any such map $f:T\to X$ has image $\bt{x,y,z}$.

Let $D$ be the triangle part of $T$ and let $x',y',z'$ as above. We have that $f$ maps $\partial D$ homeomorphically onto $\triangle(x,y,z)$, hence $f(D)=\bt{x',y',z'}$ by \cref{injective discs}. We get $f(T)=f(\partial T)\cup f(D)=\triangle(x,y,z)\cup \bt{x',y',z'}=\bt{x,y,z}$.
\end{proof}

\begin{remark}
    We have shown that in the strongly non-collinear case, our ad hoc definition coincides with the definition of the image of a minimal disc as in Lytchak and Wenger \cite{lytchak}.

\end{remark}

\section{Existence of intersection}
\label{sec: existence}
For $x_1,x_2,x_3,x_4\in X$ let $W_1,W_2,W_3,W_4$ be the four full triangles they define, where $W_i$, $1\le i \le 4$, is the full triangle of the three points $\{x_1,x_2,x_3,x_4\} - \{x_i\}$.
In \cref{non-empty intersection} we will show $\bigcap _{i=1}^4 W_i \neq \emptyset$.
If one of the $W_i$'s is a tripod or interval (i.e. the triangle part of the spiky triangle is degenerate), the claim is clear - the branching point of the tripod would be in the intersection. So we may assume this is not the case.

By \cref{triangles are discs} there exists $F_1,F_2,F_3,F_4$ full spiky triangles, $f_i:F_i \to W_i$ homeomorphism such that each side of $F_i$ maps to a side of $W_i$. The maps glue along the boundaries of the $F_i$'s and descend to a map $f$. That is, consider 
\[ S\cong \bigcup_{i=1}^4 F_i/\sim \text{ where }x\sim y \iff x\in \partial  F_i, y\in \partial F_j, f_i(x)=f_j(y).
\]

\begin{claim}
If $p \in W_i$ then there exists $j\neq i$ such that $p\in W_j$.
\end{claim}

\begin{proof}
Without loss of generality $p\in W_4$. Assume towards contradiction that $p\notin \bigcup_{i=1}^3 W_i$.
By assumption, $f( S - F_4)\subseteq X-p$. Since $S - F_4$ is contractible, $f|_{\partial F_4} : \partial F_4 \to \triangle(x_1,x_2,x_3)$ is null homotopic in $X-p$, in contradiction to the definition of $W_4$.
\end{proof}

\begin{proposition}\label{non-empty intersection}
$\bigcap_{1\leq i \leq 4} W_i \neq \emptyset$.
\end{proposition}

We will make use of Sperner's lemma, see \cite{Flegg} for the proof.

\begin{lemma}[Sperner's lemma] \label{sperner}
Consider a triangulation of a triangle with vertices $p,q,r$. Let there be a colouring of the 0-cells by three colours such that each of the three vertices $p,q,r$ has a distinct colour and the 0-cells that lie on any side of the triangle have only two colours - the two colours at the endpoints of that side. 
Then there exists at least one ``rainbow triangle'' - a 2-cell with vertices coloured with all three colours. More precisely, there must be an odd number of rainbow triangles.
\end{lemma}

\begin{figure}[htp]
    \centering
    \includegraphics[]{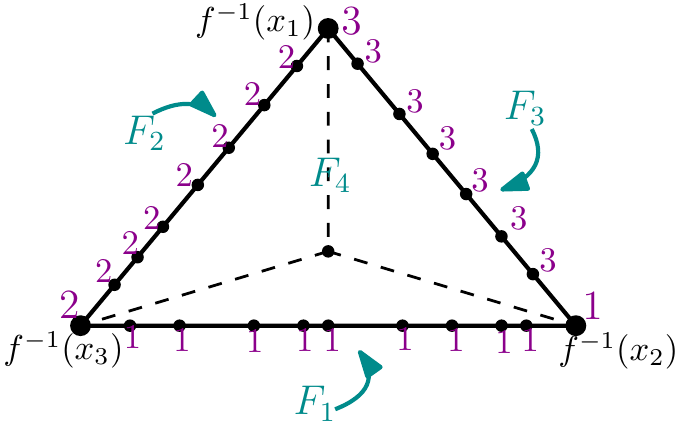}
    \caption{Colouring of the sides of $F_4$}
    \label{fig:coulouredges}
\end{figure}

\begin{proof}[Proof of \cref{non-empty intersection}]
Consider a sequence of triangulations of the triangle part of $F_4$ such that the supremum of diameters of the 2-cells goes to 0. 
We define a colouring of the 0-cells of the triangulation by three colours $1,2,3$. On $\partial F_4$ we define it as in \cref{fig:coulouredges}, i.e. for a 0-cell $v\in f^{-1}(\overline{x_2x_3})-f^{-1}(x_3)$ we let its colour be 1, the colours of 0-cells in $f^{-1}(\overline{x_1x_3})-f^{-1}(x_1)$ and $f^{-1}(\overline{x_1x_2})-f^{-1}(x_2)$ are 2 and 3 respectively. For $v$ a 0-cell in the interior of $F_4$ we define its colour to be $\min\set{j \mid f\prs{v} \in W_j}$ which is well defined from the previous claim.

By Sperner's lemma, \cref{sperner}, in each coloured triangulation there exists a rainbow triangle $T_n$. Let $p'_n$ be the barycenter of $T_n$, by compactness (up to passing to a subsequence) $p'_n$ converges to $p'\in F_4$.
We now claim that $p=f\prs{p'}$ is in $\bigcap_{1\leq i\leq 4} W_i$.

Note that for $j\in \set{1,2,3}$ the sequence of vertices of $T_n$ coloured $j$ also converges to $p'$, as the bound on the diameter of 2-cells goes to 0, hence $p'$ is arbitrarily close to $W_j$ for all $j\in \set{1,2,3}$, hence $p\in W_j$, and of course $p\in W_4$, as $p'\in F_4$, so we are done.
\end{proof}

\section{Discreteness of intersection} 
\label{sec: discreteness}
In this section we will prove the following theorem.
\begin{theorem} \label{discreteness}
Let $x_1,x_2,x_3,x_4\in X$, denote $W_1,W_2,W_3,W_4$ the four full triangles they define, then

\begin{itemize}
    \item If $x_1,x_2,x_3,x_4$ do not form a lanky quadrilateral, then $\bigcap_{i=1}^4 W_i$ is a finite set of points.
    \item If $x_1,x_2,x_3,x_4$ form a lanky quadrilateral, then $\bigcap_{i=1}^4 W_i$ is the union of a finite set and $\bigcup_{\set{i,j,k,l}=\set{1,2,3,4}}(\overline{x_ix_j}\cap \overline{x_kx_l})$, in particular it is 1-dimensional.
\end{itemize}
\end{theorem}

Let $x_1,x_2,x_3,x_4\in X$, we will refer to them as \textit{vertices} and to the geodesic segments between them as \textit{sides}. By opposite sides we mean $\overline{x_ix_j},\overline{x_kx_l}$ for $\set{i,j,k,l}=\set{1,2,3,4}$. Denote $W_1,W_2,W_3,W_4$ the four full triangles in $X$. Assume throughout this section that $p\in \bigcap_{i=1}^4 W_i$.

We will study the link of $p$ and show that the case $p\in \bigcap_{i=1}^4 W_i$ is very restricted. Indeed, denoting $\hat{x}_1,\hat{x}_2,\hat{x}_3,\hat{x}_4$ the image of $x_1,x_2,x_3,x_4$ in $\lk(p)$, we have from \cref{triangle in link 1} and \cref{triangle in link 2} that either two are antipodal, that is $p$ is on a side; or $\triangle(\hat{x}_i,\hat{x}_j,\hat{x}_k)\subseteq \lk(p)$ are well defined for all $i,j,k\in \set{1,2,3,4}$, distinct and, by \cref{homotopies in link}, not trivial in homotopy.

\begin{claim} \label{2-cell}
If $p\in \bigcap_{i=1}^4 W_i- X^{(1)}$, then $p$ is in the intersection of two opposite sides.
\end{claim}

\begin{proof}
The link of $p$ is a circle of length $2\pi$, as in \cref{fig:2-cell}. Assume by contradiction that $p$ is not on a side containing $x_i$, that is, the antipodal point to $\hat{x}_i$ is not $\hat{x}_j$ for all $j\in\set{1,2,3,4}$. 
The points $\hat{x}_i$ and its antipodal separate $\lk(p)$ into two arcs. By the pigeonhole principle, one of the arcs contains at least two out of $\set{\hat{x}_j}_{j\neq i}$, say $\hat{x}_k,\hat{x}_l$, so $\triangle(\hat{x}_i,\hat{x}_k,\hat{x}_l)$ is null homotopic in $\lk(p)$ which is a contradiction to $p\in \bt{x_i,x_k,x_l}$ and \cref{homotopies in link}, which states that $\triangle(\hat{x}_i,\hat{x}_k,\hat{x}_l)$ is homotopic to $\pi_p(\triangle(x_i,x_k,x_l))$.

In conclusion, $p$ is on a side emanating from each vertex. If $p$ is not in the intersection of two opposite sides, then up to permutation of the vertices, $p\in \bigcap_{j=2,3,4} \overline{x_1x_j}$ and in particular $\hat{x}_2=\hat{x}_3=\hat{x}_4$. That cannot be, as then $\triangle(\hat{x}_2,\hat{x}_3,\hat{x}_4)$ is null homotopic in $\lk(p)$ which is again a contradiction to \cref{homotopies in link}.

\begin{figure}[H]
    \centering
    \includegraphics[]{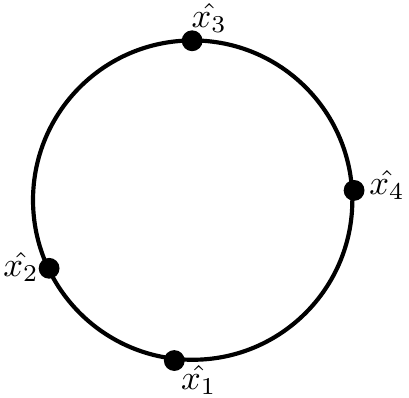}
    \caption{$\lk(p)$ for $p\notin X^{(1)}$}
    \label{fig:2-cell}
\end{figure}
\end{proof}

\begin{claim} \label{one cell}
If $p\in \bigcap_{i=1}^4 W_i$ and $p\in X^{(1)}-X^{(0)}$ is contained in the interior of a 1-cell $e$, then either $p$ is on the intersection of two opposite sides or $p$ is on a side that intersects $e$ transversely.
\end{claim}

\begin{proof}
The link $\lk(p)$ is a graph with two vertices, which we will refer to as \textit{poles} and denote $S,N$, and no loops. For the purpose of this proof, we will refer to the 1-cells in $\lk(p)$ as \textit{longitudes} in order to differentiate from the 1-cells in $X$. We may consider only the longitudes on which $\hat{x}_1,\hat{x}_2,\hat{x}_3,\hat{x}_4$ lie, as the geodesics between them will be included in the union of those longitudes. There are at most four such longitudes. Note that if there are only two such longitudes, then the same proof as in the previous claim will give us that $p$ is in the intersection of two opposite sides.
Also note that if $p$ is on a side that does not intersect $e$ transversely, then there are only two such longitudes and we are done.

Remember that each longitude is of length $\pi$ and if $p\notin \overline{x_ix_j}$ then by \cref{path in link} $\measuredangle_p(\hat{x}_i,\hat{x}_j)<\pi$. Hence if $p\notin \triangle(x_i,x_j,x_k)$, then $\triangle(\hat{x}_i,\hat{x}_j,\hat{x}_k)$ is of length $<3\pi$, so the non-backtracking representative of $\triangle(\hat{x}_i,\hat{x}_j,\hat{x}_k)$ passes through each of the poles at most once.

\textbf{Case 1. $\hat{x}_1,\hat{x}_2,\hat{x}_3,\hat{x}_4$ lie on three longitudes.} without loss of generality $\hat{x}_1,\hat{x}_2$ lie on the same longitudes and $\hat{x}_1$ is closer to $S$ than $\hat{x}_2$, as in \cref{fig:1-cell simple}.

\begin{figure}[H]
    \centering
    \includegraphics[]{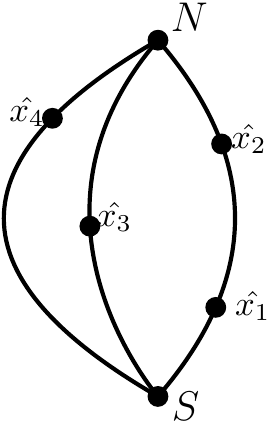}
    \caption{$\lk(p)$ for the case $p\in X^{(1)}-X^{(0)}$ and three relevant longitudes}
    \label{fig:1-cell simple}
\end{figure}

Assume towards contradiction that $p$ is not on any side, so the triangles $\set{\triangle(\hat{x}_i,\hat{x}_j,\hat{x}_k)}_{1\leq i<j<k\leq 4}$ are well defined in $\lk(p)$. They are not null-homotopic, by assumption on $p$.

We have that $S\in\overline{\hat{x}_1\hat{x}_3}$, as if not, $N\in\overline{\hat{x}_1\hat{x}_3}$. Hence $N\in\overline{\hat{x}_2\hat{x}_3}$, as a subset of a geodesic is a geodesic. It follows that $\triangle(\hat{x}_1,\hat{x}_2,\hat{x}_3)$ is null homotopic which is a contradiction to $p\in\bt{x_1,x_2,x_3}$. 

Similar reasoning replacing 3 by 4 gives $S\in \overline{\hat{x}_1\hat{x}_4}$. Similar reasoning replacing 1 by 2 and $S$ by $N$ gives $N\in \overline{\hat{x}_2\hat{x}_3},\overline{\hat{x}_2\hat{x}_4}$. 

We arrive to a contradiction, as the geodesic between $\hat{x}_3,\hat{x}_4$ does not pass through one of the poles, say $S$, hence the triangle $\triangle(\hat{x}_2,\hat{x}_3,\hat{x}_4)$ is null homotopic - does not pass through $S$.

\textbf{Case 2. $\hat{x}_1,\hat{x}_2,\hat{x}_3,\hat{x}_4$ lie on four longitudes.} It cannot be that in $\lk(p)$ all geodesics emanating from a vertex pass through the same pole.
Indeed, assume towards contradiction and without loss of generality that all geodesics emanating from $\hat{x}_1$ pass through $S$. We get that all other geodesics pass through $N$ because $\triangle(\hat{x}_1,\hat{x}_i,\hat{x}_j)$ is not null homotopic for $2\leq i<j\leq 4$. But then $\triangle(\hat{x}_2,\hat{x}_3,\hat{x}_4)$ is null homotopic, as all the geodesic segments in it pass through $N$, so do not pass through $S$. This is a contradiction to $p\in \bt{x_2,x_3,x_4}$.

Assume without loss of generality that $\hat{x}_1$ is the closest to $S$ and $\hat{x}_4$ is the closest to $N$ (need not be unique), as in \cref{fig:1-cell}. If the geodesic between $\hat{x}_1,\hat{x}_4$ passes through $S$ (respectively $N$) then all geodesics emanating from $\hat{x}_1$ (respectively $\hat{x}_4$) pass through $S$ (respectively $N$). So all geodesics emanating from $\hat{x}_1$ or $\hat{x}_4$ pass through the same pole which is a contradiction by the previous paragraph.

\begin{figure}[H]
    \centering
    \includegraphics[]{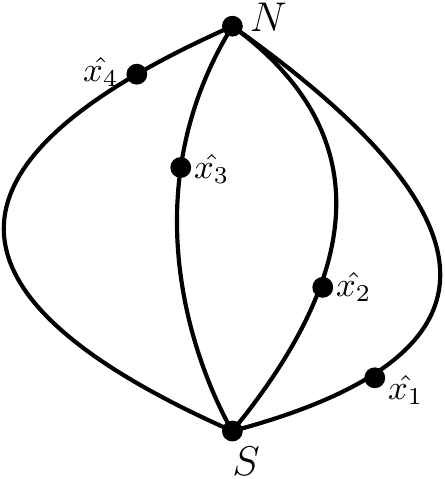}
    \caption{$\lk(p)$ for the case $p\in X^{(1)}-X^{(0)}$ and four relevant longitudes}
    \label{fig:1-cell}
\end{figure}
\end{proof}

\begin{proof}[Proof of \cref{discreteness}] 
By \cref{2-cell} and \cref{one cell}, every point of $\bigcap_{i=1}^4 W_i$ is either in $X^{(0)}$, in the transverse intersection of $X^{(1)}$ with sides, or in the intersection of sides. 

We have that each side crosses transversely a finite amount of 1-cells in $X$, $X^{(0)} \cap \bigcup_{i=1}^4 W_i$ is finite, and if $x_1,x_2,x_3,x_4$ do not form a lanky quadrilateral, two sides intersect at most at a point, so $\bigcap_{i=1}^4 W_i$ is finite.
If they do, we have that $\bigcap_{i=1}^4 W_i$ is a union of a finite set of points with all intersections of opposite sides.
\end{proof}

This theorem is essential to the proof, it will be used together with the following lemma, regarding the intersection of two triangles sharing a side.

\begin{lemma} \label{intersection of two triangles}
Assume $x_1,x_2,x_3,x_4\in X$, $x_1\neq x_2$ and $q\in \bt{x_1,x_2,x_3}\cap \bt{x_1,x_2,x_4}$, then $q$ is not an isolated point in the intersection.
\end{lemma}

\begin{proof}
Let $\hat{x}_1,\hat{x}_2,\hat{x}_3,\hat{x}_4$ be the images of $x_1,x_2,x_3,x_4$ in $\lk(q)$.

Note that if $q\in \overline{x_1x_2}$ we are done, as $\overline{x_1x_2}\subseteq \bt{x_1,x_2,x_3}\cap \bt{x_1,x_2,x_4}$ so $q$ is not an isolated point in the intersection. Assume this is not the case.

By \cref{triangle in link 1} and \cref{triangle in link 2} it is enough to show that $C_3,C_4$, the essential parts of $\pi_q(\triangle(x_1,x_2,x_3)-q),\pi_q(\triangle(x_1,x_2,x_4)-q)$ in $\lk(q)$ respectively, intersect. 
We will actually show that they intersect on $\overline{\hat{x}_1\hat{x}_2}$. 

Clearly $C_3,C_4$ intersect $\overline{\hat{x}_1\hat{x}_2}$. If for some $i\in\set{3,4}$, $q\in \overline{x_1x_i}\cap\overline{x_2x_i}$ we have by \cref{homotopies in link} that $\overline{\hat{x}_1\hat{x}_2}\subseteq C_i$ and so $C_3\cap C_4\neq \emptyset$ and we are done. So assume $q\notin  \overline{\hat{x}_1\hat{x}_3}\cap  \overline{\hat{x}_2\hat{x}_3}, \overline{\hat{x}_1\hat{x}_4}\cap  \overline{\hat{x}_2\hat{x}_4}$.
In this case, again by \cref{homotopies in link}, each of $C_3,C_4$ is a cycle or an arc depending on whether $q\in \triangle(x_1,x_2,x_i)$ or not.

If $C_3$ is a cycle, then from \cref{triangles are spiky in link} $C_3=\triangle(\hat{z}_1,\hat{z}_2,\hat{z}_3)$ where $\hat{z}_i$ is the point closest to $\hat{x}_i$ on $C_3$, that is it is the point furthest away from $\hat{x}_i$ on $\overline{\hat{x}_i\hat{x}_j}\cap \overline{\hat{x}_i\hat{x}_k}$ where $\set{i,j,k}=\set{1,2,3}$.

Similarly if $C_4$ is a cycle then $C_4=\triangle(\hat{w}_1,\hat{w}_2,\hat{w}_4)$ where $\hat{w}_i$ is the point closest to $\hat{x}_i$ on $C_4$, that is it is the point furthest away from $\hat{x}_i$ on $\overline{\hat{x}_i\hat{x}_j}\cap \overline{\hat{x}_i\hat{x}_k}$ where $\set{i,j,k}=\set{1,2,4}$.

If $C_3$ is an arc then $q\in \overline{x_1x_3}$ or $\overline{x_2x_3}$ and $C_3=\overline{\hat{x}_1\hat{z}_3}\cup \overline{\hat{x}_1\hat{z}_3}$ where $\hat{z}_3$ is defined as above.
Similarly if $C_4$ is an arc $C_4=\overline{\hat{x}_1\hat{w}_4}\cup \overline{\hat{x}_1\hat{w}_4}$

In all cases $C_i\cap \overline{\hat{x_1}\hat{x_2}}$ is non-trivial and connected, hence an arc (perhaps degenerate, i.e. a point).

Assume towards contradiction that $C_3,C_4$ do not intersect on $\overline{\hat{x}_1\hat{x}_2}$, that is $C_3\cap \overline{\hat{x}_1\hat{x}_2},C_4\cap \overline{\hat{x}_1\hat{x}_2}$ are two disjoint arcs.
Without loss of generality $C_3\cap \overline{\hat{x}_1\hat{x}_2}$ is closer to $\hat{x}_1$ than $C_4\cap \overline{\hat{x}_1\hat{x}_2}$. The endpoint of $C_3\cap \overline{\hat{x}_1\hat{x}_2}$ furthest from $\hat{x}_1$ is $\hat{z}_2$ and the endpoint of $C_4\cap \overline{\hat{x}_1\hat{x}_2}$ furthest from $\hat{x}_2$ is $\hat{w}_1$, as in \cref{fig:tr}. By assumption $\overline{\hat{x}_1\hat{z}_2}\cap \overline{\hat{w}_1\hat{x}_2}=\emptyset$.

\begin{figure}[H]
    \centering
    \includegraphics{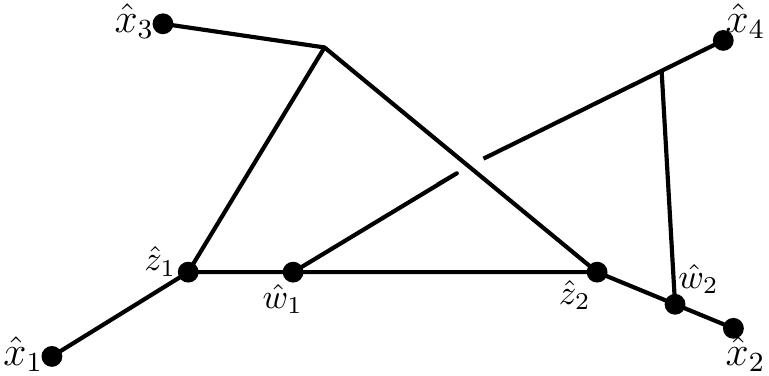}
    \caption{$\triangle(\hat{x}_1,\hat{x}_2,\hat{x}_3),\triangle(\hat{x}_1,\hat{x}_2,\hat{x}_4)$ in $\lk(q)$}
    \label{fig:tr}
\end{figure}

Let $\alpha=d(\hat{w}_1,\hat{x_2})$ and $\beta=d(\hat{z}_2,\hat{x_1})$.

Note that:

\begin{itemize}
    \item $d(\hat{w}_1,\hat{x}_4)\geq \pi-\alpha$.
    
    This is because:
    \begin{itemize}
        \item If $q\notin \overline{x_2x_4}$ then the concatenation of $\overline{\hat{x_2}\hat{w}_1},\overline{\hat{w}_1\hat{x_4}},\overline{\hat{x_4}\hat{x_2}}$ includes an essential cycle, so is of length $\geq 2\pi$, but $\overline{\hat{x}_2\hat{x}_4}$ is of length $<\pi$ by \cref{path in link}, so $\alpha+d(\hat{w}_1,\hat{x}_4)\geq 2\pi-d(\hat{x}_2,\hat{x}_4)>\pi$ and this gives us what we want.
        
        \item If $q\in \overline{x_2x_4}$ we can choose a path $\hat{\pi}$ of length $\pi$ between $\hat{x}_4,\hat{x}_2$ and the concatenation of $\overline{\hat{x}_2\hat{w}_1},\overline{\hat{w}_1\hat{x}_4},\hat{\pi}$ either contains a cycle and then, as in the previous case, we get what we want or is fully backtracking and then $\alpha+d(\hat{w}_1,\hat{x}_4)\geq\pi$.
    \end{itemize}
    \item $d(\hat{x}_1,\hat{w}_1) + d(\hat{w}_1,\hat{x}_4) < \pi$.
    
    This is because the concatenation of $\overline{\hat{x}_1\hat{w}_1},\overline{\hat{w}_1\hat{x}_4}$ is exactly $\overline{\hat{x}_1\hat{x}_4}$, by choice of $\hat{w}_1$, so we are done by \cref{path in link}.
\end{itemize}  

From the second point we have $\beta=d(\hat{x}_1,\hat{z}_2) \leq d(\hat{x}_1,\hat{w}_1) < \pi-d(\hat{w}_1,\hat{x}_4)$, and then from the first point $\beta<\alpha$. But of course, this argument was symmetric so we also have $\alpha<\beta$ which is a contradiction.
\end{proof}

\section{Proof of Theorem for lanky quadrilateral case}
\label{sec: 1d case}
The statement we will prove in this section is the second part of \cref{main theorem}:
\begin{theorem} \label{1-dim case}
Assume $x_1,x_2,x_3,x_4\in X$ such that $\overline{x_1x_2}\cap \overline{x_3x_4}$ is a segment, then $\bigcap_{1\leq i<j<k\leq 4}\bt{x_i,x_j,x_k}=\overline{x_1x_2}\cap \overline{x_3x_4}$ 
\end{theorem}

The situation described in the theorem corresponds to triangles whose union is as in \cref{fig:1-dim case triangles}.

\begin{figure}[H]
    \centering
    \includegraphics[]{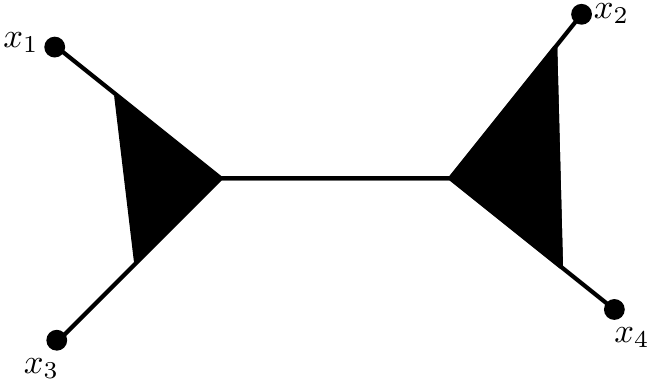}
    \caption{$\bigcup_{1\leq i<j<k\leq 4}\bt{x_i,x_j,x_k}$ when $x_1,x_2,x_3,x_4$ form a lanky quadrilateral}
    \label{fig:1-dim case triangles}
    \end{figure}

\begin{lemma} \label{square or butterfly}
Given $x_1,x_2,x_3,x_4\in X$, if $\overline{x_1x_2}\cap \overline{x_3x_4}\neq \emptyset$ then 
\[
    \overline{x_1x_3}\cap \overline{x_2x_4},
    \overline{x_1x_4}\cap \overline{x_2x_3}\subseteq \overline{x_1x_2}\cap \overline{x_3x_4},
    \]
hence $\overline{x_ix_j}\cap \overline{x_kx_l}= \overline{x_{i'}x_{j'}}\cap \overline{x_{k'}x_{l'}}$ for $\set{i,j,k,l}=\set{i',j',k',l'}=\set{1,2,3,4}$, when both intersections are non-empty.
\end{lemma}

\begin{proof}
We will prove $\overline{x_1x_3}\cap \overline{x_2x_4}\subseteq \overline{x_1x_2}\cap \overline{x_3x_4}$, the proof of the other inclusion is the same up to renaming.

By assumption there exists $p\in\overline{x_1x_2}\cap \overline{x_3x_4}$. Let $q\in\overline{x_1x_3}\cap \overline{x_2x_4}$.

Things are now as in \cref{fig:Figure}, where lines of the same colour are geodesics and the letters in gray are the lengths of each part of the geodesic.

\begin{figure}[htp]
    \centering
    \includegraphics[]{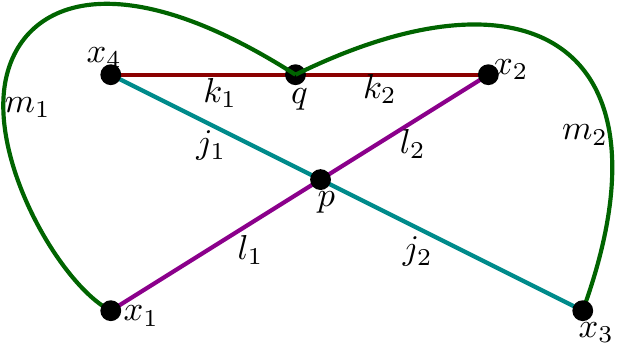}
    \caption{The geodesic segments between $x_i,x_j$}
    \label{fig:Figure}
    \end{figure}

By the triangle inequality, we have the following
\begin{align*}
    l_1+l_2&\le m_1+k_2 \\
    j_1+j_2 & \leq m_2+k_1 \\
    k_1+k_2 &\leq j_1+l_2 \\
    m_1+m_2 &\leq l_1+j_2.
\end{align*}
Summing up these inequalities we get $l_1+l_2+k_1+k_2+m_1+m_2+j_1+j_2\leq l_1+l_2+k_1+k_2+m_1+m_2+j_1+j_2$, so all inequalities are equalities.

Geodesics are unique in a CAT(0) space so we get 
\begin{align*}
    \overline{x_1q}\cup \overline{qx_2}&= \overline{x_1x_2} \\
    \overline{x_4q}\cup\overline{qx_3} & = \overline{x_4x_3}
\end{align*}

In particular $q\in \overline{x_1x_2}\cap\overline{x_3x_4}$.
\end{proof}

\begin{lemma} \label{intersection}
In the conditions of \cref{1-dim case}, $\bt{x_1,x_3,x_4}\cap \bt{x_2,x_3,x_4}$ is equal to $\bigcap_{1\leq i<j<k\leq 4}\bt{x_i,x_j,x_k}$.
\end{lemma}

\begin{proof}
It suffices to show $\bt{x_1,x_3,x_4}\cap \bt{x_2,x_3,x_4}\subseteq \bigcap_{1\leq i<j<k\leq 4}\bt{x_i,x_j,x_k}$.

 Denote by $y_1,y_2$ the two endpoints of $\overline{x_1x_2}\cap\overline{x_3x_4}$ such that $y_1$ is closer to $x_1$ than to $x_2$ and. Without loss of generality, $y_1$ is also closer to $x_3$ than to $x_4$, as in \cref{fig:1-dim case} .

\begin{figure}[htp]
    \centering
    \includegraphics{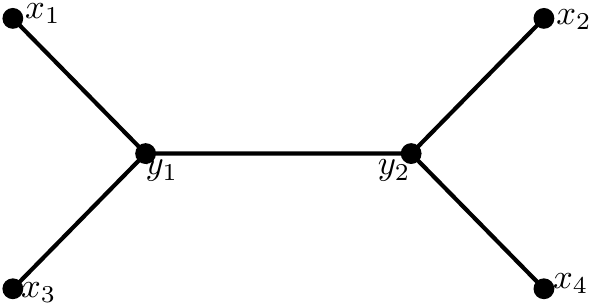}
    \caption{$\overline{x_1x_2}\cup\overline{x_3x_4}$}
    \label{fig:1-dim case}
    \end{figure}

Note that $\overline{x_4x_1}=\overline{x_4y_1}\cup\overline{y_1x_1}$ because $\overline{x_4y_1},\overline{y_2x_1}$ are both geodesics that intersect on an interval. Similarly $\overline{x_2x_3}=\overline{x_2y_1}\cup\overline{y_1x_3}$. Hence for $i\in\set{2,4}$ we have $\bt{x_1,x_i,x_3}=\overline{x_iy_1}\cup \bt{y_1,x_1,x_3}$ and similarly for $j\in \set{1,3}$ we have $\bt{x_2,x_j,x_4}=\overline{x_jy_2}\cup \bt{y_2,x_2,x_4}$. 
So \[\bt{x_1,x_3,x_4}\cap \bt{x_2,x_3,x_4}=(\overline{x_4y_1}\cup \bt{y_1,x_1,x_3})\cap (\overline{x_3y_2}\cup \bt{y_2,x_2,x_4}).\]

It is enough to show $\bt{y_1,x_1,x_3})\cap \bt{y_2,x_2,x_4})\subseteq \overline{x_1x_2}\cap \overline{x_3x_4}$. If $p\in \bt{x_1,y_1,x_3}$, then $\pi_p(x_2)=\pi_p(x_4)$, so $p\notin \bt{y_2,x_1,x_4}-\triangle(y_2,x_1,x_4)$. Similarly interchanging the roles of the triangles. So \begin{align*}
    \bt{y_1,x_1,x_3})\cap \bt{y_2,x_2,x_4})&=\triangle(y_1,x_1,x_3)\cap \triangle(y_2,x_2,x_4) \\
    &\subseteq (\overline{x_1x_3}\cap \overline{x_2x_4}) \cup
    (\overline{x_1x_4}\cap \overline{x_2x_3}),
\end{align*} which is in $\subseteq \overline{x_1x_2}\cap \overline{x_3x_4}$ by \cref{square or butterfly}.
\end{proof}

Now for the proof of the theorem. 

\begin{proof}[Proof of \cref{1-dim case}]
Denote $T_1=\bt{x_1,x_3,x_4},T_2=\bt{x_2,x_3,x_4}$.
By \cref{intersection} $T_1\cap T_2=\bigcap 
 _{1\leq i<j<k\leq 4}\bt{x_i,x_j,x_k}$, which by \cref{discreteness} is equal to $\bigcup_{\set{i,j,k,l}=\set{1,2,3,4}} (\overline{x_ix_j}\cap \overline{x_kx_l})$ union a finite set of points. By \cref{intersection of two triangles} $T_1\cap T_2$ has no isolated points, thus $T_1\cap T_2=\bigcup_{\set{i,j,k,l}=\set{1,2,3,4}} (\overline{x_ix_j}\cap \overline{x_kx_l})$ and by \cref{square or butterfly}, all none empty segments among these coincide.

\end{proof}

\section{Near-immersions}
\label{sec: near-immersions}

Recall \cref{near-immersion}. We rewrite the definition for 2-dimensional complexes.

\begin{definition}
A cellular map $f:A\to B$ between 2-dimensional polygonal complexes is a \textit{near-immersion} if it is injective on each cell and
there are no \emph{twin cells} - different 2-cells $t,t'$ sharing a 1-cell such that $f(t)=f(t')$. 
\end{definition}

\begin{theorem}\label{local link}
Let $X$ be a CAT(0) polygonal complex, and let $\Gamma$ be an embedded polygonal curve with ordered vertices $x_1,...,x_n$ such that $\kappa(\Gamma)<4\pi$. Assume $f:\Delta \to X$ is a near-immersion, where $\Delta$ is a finite $\Delta$-complex homeomorphic to a disc and $f|_{\partial \Delta} : \partial \Delta \to \Gamma$ is a homeomorphism, then $f$ is injective.

In particular, $\Gamma=\triangle(x,y,z)$ satisfies the conclusion of the theorem when $x,y,z\in X$ are strongly non-collinear, as in this case $\kappa(\Gamma)<4\pi$.
\end{theorem}

We get the following corollary using \cref{triangle inj implies spiky triangle}.

\begin{corollary} \label{local link for spiky}
    Let $X$ be a CAT(0) polygonal complex, $x,y,z\in X$ and $f:T \to X$ a near-immersion such that $T$ is a full (abstract) spiky triangle and $f|_{\partial T} : \partial T \to \triangle(x,y,z)$ is a homeomorphism.
Then $f$ is injective and its image is $\bt{x,y,z}$.
\end{corollary}

\begin{proof}
By \cref{local link}, $f$ is injective on the triangle part of $T$, hence by \cref{triangle inj implies spiky triangle} $f$ is injective. By \cref{triangles are discs} $f(T)=\bt{x,y,z}$.
\end{proof}

We will prove \cref{local link} by arguing that if $f$ is not injective we get a contradiction to Gauss-Bonnet Theorem. 

We will remind the reader of the statement of Gauss-Bonnet.

\begin{definition}
A \emph{pseudo-surface with boundary} is a 2-dimensional simplicial complex $S$ such that every 1-cell is contained in one or two 2-cells. Its \emph{boundary}, $\partial S$, is the union of the 1-cells which are incident to only one 2-cell. Its \emph{interior}, $\inner(S)$, is defined to be $S-\partial S$.
\end{definition}

\begin{notation}
Let $S$ be a piecewise Euclidean pseudo-surface.
For $v\in \inner(S)$ we denote $K_v=2\pi - \sum \alpha$ where the sums run over all angles around $v$ in the 2-cells containing $v$. For $v\in \partial S$, let $i$ be the number of connected component of $\lk(v)$; we denote $k_v = (2-i)\pi - \sum \alpha$ where the sums run over all angles around $v$ in the 2-cells containing $v$.
\end{notation}

\begin{theorem} [\cite{Fansandladders} Theorem 4.6] \label{GB}
Let $S$ be a piecewise Euclidean pseudo-surface. Then,
\[
\sum_{v\in \inner(S)} K_v+\sum_{v\in \partial S} k_v=2\pi \chi (S)
\]
where $\chi (S)$ is the Euler characteristic of $S$.

Note that the sums are finite, as for $v\notin S^{(0)}$ we have that the appropriate summand is 0.
\end{theorem}

In the rest of this section we will prove \cref{local link}. 

Assume $\Delta$ is a finite $\Delta$-complex homeomorphic to a disc and $f:\Delta\to X$ is a near-immersion mapping the boundary of $\Delta$ homeomorphically onto $\Gamma$. 
We endow $\Delta$ with a piecewise Euclidean structure by endowing each 2-cell $\sigma$ of $\Delta$ with the Euclidean metric of its image in $X$ under $f$. Let $p\in\Delta$, since $f$ is a near-immersion, each arc $\alpha$ in $\lk(p)$ is sent to an arc $f_p(\alpha)$ in $\lk(f(p))$ which has the same length. It follows that the girth of $\lk(p)$ is at least that of $\lk(f(p))$, which is at least $2\pi$ since $X$ is CAT(0). By the link condition, \cref{link condition}, $\Delta$ is CAT(0).

Denote abusively by $x_i\in \Delta$ the unique point in $\partial\Delta$ with image $x_i\in X$. 

We collect a few important facts about $\Delta$ in the following claims.

\begin{claim}\label{claim about Delta} 
We have,
\begin{enumerate}[label = (\arabic*)]
    \item for all $v\in\inner(\Delta)$, $K_v\le 0$,  \label{claim about Delta. interior points}
    \item for all $v\in\partial\Delta-\set{x_1,...,x_n}$,  $k_v\le 0$,  \label{claim about Delta. boundary points}
\end{enumerate}
\end{claim}

\begin{proof}
\underline{Proof of (1):} Since $\Delta$ is CAT(0), the girth of $\lk(v)$ is at least $2\pi$. If $v\in\inner \Delta$ then $\lk(v)$ is a circle, and $K_v = 2\pi - \girth(\lk (v)) \le 0$, as claimed.

\underline{Proof of (2):} If $v\in \inner(\gamma) \cap \partial\Delta-\set{x_1,...,x_n}$, then $f(v)$ is on the geodesic segment connecting  $x_i,x_{i+1}$ or $x_n,x_1$. By the triangle inequality for angles, the sum of the angles of $v$ in 2-cells containing $v$ is at least $\angle_v(x_i,x_{i+1})=\pi$ (or, in the latter case, $\angle_v(x_n,x_1)=\pi$), hence $k_v \leq 0$.
\end{proof}

Let $\gamma$ be a geodesic in $\Delta$ connecting the points $p,q$. For $v\in \inner(\gamma)$ define
$$T_v = \pi - \angle_{f(v)}(f\circ \gamma)$$ where $\angle_{f(v)}(f\circ \gamma)$ is the angle that $f\circ \gamma$ makes at $f(v)$ \textbf{in $X$}. By definition, $0\le T_v\le \pi$.

\begin{claim}\label{inequalities claim} Assume $\gamma$ is a geodesic in $\Delta$ such that $f\circ \gamma$ is a closed loop in $X$ Then:
\begin{enumerate}[label=(\roman*)]
    \item $\sum_{v\in \inner(\gamma)} T_v \ge \pi$.\label{first point}
    
    \item If $v\in \inner(\gamma) \cap \inner(\Delta)$ then $K_v\le -2 T_v$. \label{second point}
    
    \item If $v\in \inner(\gamma)\cap \set{x_1,\dots ,x_n}$ then $k_v\leq -2T_v+\pi-\angle_{x_i}\Gamma$, where $T_v = \pi-\angle_{f(v)}f\circ \gamma$.  \label{third point}
\end{enumerate}
\end{claim}

\begin{proof}
To prove \ref{first point}, we will use the following lemma.
    \begin{lemma} \label{polygon lemma}
    Let $P$ be a geodesic polygon in a CAT(0) space with ordered vertices $v_1,\dots,v_n$ then $\sum_{i=1}^n (\pi-\angle_{v_i}P)\geq 2\pi$.\qed
    \end{lemma}

    The proof of \cref{polygon lemma} is standard and is obtained by cutting the polygon into triangles by erecting geodesic segments between $v_i$ for $2\leq i\leq n-2$ and $v_n$ and using the fact that the sum of angles in a triangle in a CAT(0) space is less than or equal to $\pi$ (\cite{BridHäf} Proposition II.1.7).
    See for example Theorem 2.4 in \cite{DisDif2008} and Corollary 2.4 in \cite{AlexBish}.
    
    Consider the geodesic polygon $P$ which is the image of $f\circ \gamma$. We have $\sum_{v\in \inner(\gamma)} T_v = \sum_{w\in P} (\pi - \angle_{w} P) - (\pi - \angle_{f(\gamma(0))} P) \ge 2\pi - \pi=\pi$ where the last inequality follows from \cref{polygon lemma}. This proves \ref{first point}.
    
    \bigskip
    
    To prove \ref{second point}, let $v\in \inner(\gamma) \cap \inner(\Delta)$.
    If $T_v =0$ then by \cref{claim about Delta}\ref{claim about Delta. interior points} $K_v \le 0  = T_v$, therefore we may assume that $T_v>0$, i.e. the angle that the path $f\circ \gamma$ makes at $f(v)$ in $X$ is strictly less than $\pi$.
    
    Let $\gamma_+$ and $\gamma_-$ be the two points in $\lk(v)$ corresponding to the geodesic $\gamma$ passing through $v$.
    The points $\gamma_+,\gamma_-$ break the circle $\lk(v)$ into two arcs $\alpha_1,\alpha_2$, as in the left hand side of \cref{fig:link 1}.
    As the girth of $\lk(v)$ is greater than $2\pi$, it follows that $\measuredangle \alpha_1+\measuredangle \alpha_2\geq 2\pi$.
    Since $f$ is a near-immersion, the paths $f_v\circ \alpha_i$ are non-backtracking in $\lk(f(v))$ and $\measuredangle \alpha_i = \measuredangle (f_v\circ \alpha_i)$.
    
    From \cref{path in link} there is a path $\beta$ in $\lk(f(v))$ whose length realizes the angle that $f\circ \gamma$ makes at $f(v)$. In $\lk(f(v))$, all three paths $\beta, f_v \circ \alpha_1, f_v \circ \alpha_2$ are non-backtracking and connect $f_v(\gamma_+),f_v(\gamma_-)$, as in the right hand side of \cref{fig:link 1}.
    
    \begin{figure}[htp]
    \centering
    \includegraphics[]{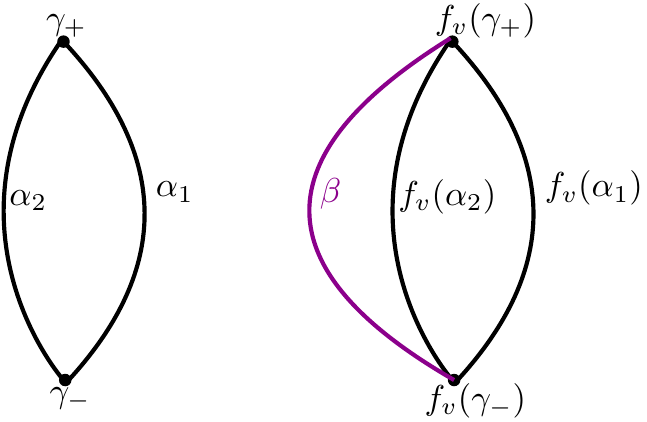}
    \caption{The links of $v$ (on the left) and $f(v)$ (on the right)}
    \label{fig:link 1}
    \end{figure}
    
    For $i\in \set{1,2}$, $\measuredangle \alpha_i \ge \angle(\gamma_-,\gamma_+) = \pi$, since $\gamma$ is a geodesic in $\Delta$, hence in particular $\measuredangle (f_v \circ \alpha_i) = \measuredangle \alpha_i \ge \pi >\measuredangle \beta$, so the concatenation of $f_v\circ \alpha_i$ and $\beta$ forms an essential cycle in $\lk(f(v))$. Hence $\measuredangle (f_v \circ \alpha_i)+\measuredangle \beta\geq 2\pi$ which implies $\measuredangle \alpha_i\geq 2\pi-\measuredangle \beta$.
    
    In conclusion $$ K_v=2\pi-(\measuredangle \alpha_1+\measuredangle \alpha_2)\leq 2\pi - 2\cdot(2\pi-\measuredangle \beta)=-2(\pi-\measuredangle \beta)=-2T_v,$$ as claimed. 
    
    \bigskip
    
    To prove \ref{third point}, let $v\in \inner(\gamma) \cap \set{x_1,\dots ,x_n}$. 
We must have $k_v\leq 0$, because otherwise, we could find a shortcut of $\gamma$ not passing through $v$, in contradiction to $\gamma$ being a geodesic. Hence if $T_v=0$, then $k_v\le 0 = T_v$, as claimed. Therefore, we may assume that $T_v >0$, i.e. that the angle that the path $f\circ \gamma$ makes at $f(v)$ in $X$ is less than $\pi$.
    
As in the proof of \cref{inequalities claim} \ref{third point}, let $\gamma_+,\gamma_-$ be the two points in $\lk(v)$ corresponding to $\gamma$. They break the arc $\lk(v)$ into three arcs $\alpha_1,\alpha_2,\alpha_3$ where $\alpha_2$ is the middle arc and $\alpha_1,\alpha_3$ may be degenerate, as in the left-hand side of \cref{fig:link 2}. Since $f$ is a near-immersion, the paths $f_v\circ \alpha_i$ are non-backtracking in $\lk(f(v))$ and $\measuredangle \alpha_i = \measuredangle (f_v\circ \alpha_i)$.

From \cref{path in link} there is a path $\beta$ in $\lk(f(v))$ whose length realizes the angle that the path $f\circ \gamma$ makes at $f(v)$.
    We have that $\beta,f_v\circ \alpha_2$ connect $f_v(\gamma_+),f_v(\gamma_-)$ while $f_v\circ\alpha_1,f_v\circ \alpha_3$ connect to $f_v(\gamma_+),f_v(\gamma_-)$ respectively, as in the right hand side of \cref{fig:link 2}.

    We have $\measuredangle \alpha_2\geq \angle(\gamma_-,\gamma_+)=\pi$ where the last equality follows, since $\gamma$ is a geodesic, hence in particular $\measuredangle(f_v \circ \alpha_2)=\measuredangle\alpha_2\geq \pi > \measuredangle \beta$, so the concatenation of $f_v\circ \alpha_2$ and $\beta$ forms an essential cycle in $\lk(f(v))$. We get $\measuredangle (f_v \circ \alpha_2)+\measuredangle \beta\geq 2\pi$ which implies $\measuredangle \alpha_2\geq 2\pi-\measuredangle \beta$.

    \begin{figure}[htp]
    \centering
    \includegraphics[]{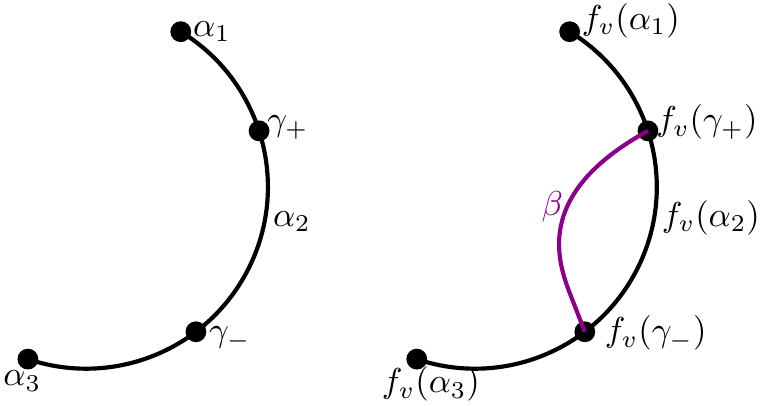}
    \caption{The links of $v$ (on the left) and $f(v)$ (on the right)}
    \label{fig:link 2}
    \end{figure}

So $k_v=\pi - (\measuredangle\alpha_1+\measuredangle\alpha_2+\measuredangle\alpha_3)\leq \pi -(2\pi-\measuredangle \beta)-(\measuredangle\alpha_1+\measuredangle\alpha_3)=-T_v-(\measuredangle\alpha_1+\measuredangle\alpha_3)=-2T_v+(\pi-(\measuredangle\alpha_1+\measuredangle \beta+\measuredangle\alpha_3)\leq -2T_v + (\pi-\angle_v \Gamma)$

where the last inequality is from the triangle inequality for angles.

    This completes the proof of the claim.
\end{proof}

We are now ready to prove \cref{local link}.

\begin{proof}[Proof of \cref{local link}]

Let us assume towards contradiction that $f$ is not injective. Therefore there are $p\neq q\in \Delta$ such that $f(p)=f(q)$ and consider $\gamma:\brs{0,1} \to \Delta$ a geodesic between them in $\Delta$. Then $f\circ \gamma:\brs{0,1}\to X$ is a non-trivial, piecewise linear loop. We choose $p,q$ closest such points, which ensures that $f\circ \gamma$ is a simple closed curve. 

\begin{equation}
        \begin{split}
    2\pi & \overset{(1)}{=}\sum_{v\in \inner(D)} K_v+\sum_{v\in \partial(D)} k_v \\
    & =\sum_{v\in \inner(\gamma) \cap \inner(D)} K_v+\sum_{v\in \inner(\gamma) \cap \partial(D)} k_v + \sum_{v\in \inner(D)-\inner(\gamma)} K_v + \sum_{v\in \partial(D)-\inner(\gamma)}k_v \\ &
    \overset{(2)}{\leq} -2(\sum_{v\in \inner(\gamma) \cap \inner(D)} T_v+\sum_{v\in \inner(\gamma) \cap \partial(D)} T_v) +\sum_{x_i\in \inner(\gamma)}(\pi-\angle_{x_i}\Gamma) \\&
    + \sum_{v\in \inner(D)-\inner(\gamma)} K_v + \sum_{v\in \partial(D)-\inner(\gamma)}k_v \\&
    \overset{(3)}{\le} -2(\sum_{v\in\inner(\gamma)}T_v) +\sum_{x_i\in \inner(\gamma)}(\pi-\angle_{x_i}\Gamma) +\sum_{x_i\notin \inner(\gamma)}k_{x_i} \\ &
    \overset{(4)}{<} -2\pi +4\pi -\sum_{x_i\notin \inner(\gamma)}(\pi-\angle_{x_i}\Gamma) +\sum_{x_i\notin \inner(\gamma)}k_{x_i}  \overset{(5)}{\leq}2\pi 
    \end{split}
\end{equation}
Where (1) is \cref{GB}, (2) from \cref{inequalities claim}\ref{second point},\cref{inequalities claim}\ref{third point}, (3) from the fact that $K_v\leq 0$ and $k_v\leq 0$ for $v\notin\set{x_1,\dots ,x_n}$, (4) from $\sum T_v\geq \pi$ (\cref{inequalities claim}\ref{first point}) and $\kappa(\Gamma)<4\pi$, and finally (5) follows from $k_{x_i}+\pi-\angle_{x_i}\Gamma\leq 0$.
This is a contradiction.
\end{proof}

\section{Minimal tetrahedron}
\label{sec: minimal tetrahedron}
The idea of the proof of the main theorem is to show that there is an embedding of a ``deflated tetrahedron'', as in \cref{fig:def tetra}, to $X$, such that the sides are mapped to $\overline{x_ix_j}$, because then the ``faces'' of the deflated tetrahedron must be mapped to $\bt{x_i,x_j,x_k}$, and they intersect in a unique point. 

\begin{figure}[H]
    \centering
    \includegraphics[]{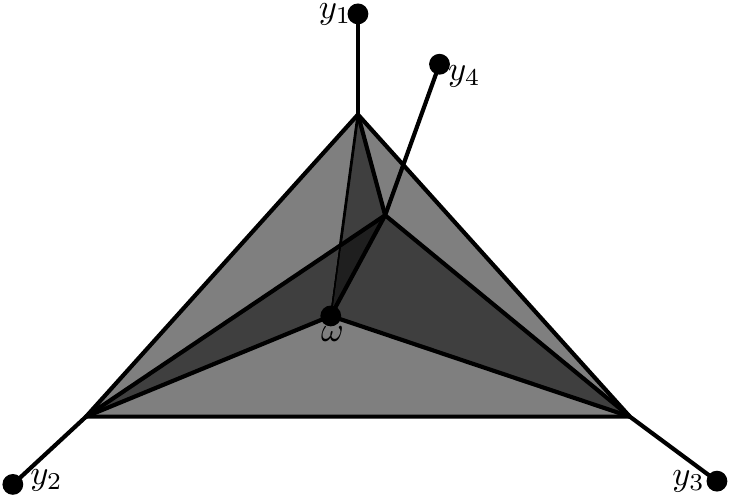}
    \caption{Deflated tetrahedron}
    \label{fig:def tetra}
    \end{figure}

In our minds, we start with a finite $\Delta$-complex, as in \cref{fig:def tetra}, and a map to $X$, and we modify those until we get an embedding. This leads us to our next definition, keeping in mind that in \cref{fig:def tetra} we have six spiky triangles, $\set{\bt{y_i,y_j,\omega}}_{i\leq i<j\leq 4}$, glued on their boundaries.

\begin{definition}
\begin{itemize}

 \item A \emph{deflated tetrahedron} is a collection of:
 \begin{itemize}
     \item six generalized triangles (which are allowed to be degenerate, recall \cref{generalized disc and triangle}) $\set{D_{ij}}_{1\leq i<j\leq 4}$ with the marked points on $D_{ij}$ denoted $y_i,y_j,\omega$, (denoting $D_{ji} = D_{ij}$ when $i<j$),
     \item bijections between $\overline{y_i\omega}$ in $D_{ij}$ and $\overline{y_i\omega}$ in $D_{ik}$,
 \end{itemize}
   such that the complex $T_{ijk}$, obtained by gluing $D_{ij}, D_{jk},D_{ik}$ along their boundaries via the bijections above, is homeomorphic to a spiky triangle. 
   
   For a deflated tetrahedron $\set{D_{ij}}_{1\leq i<j\leq 4}$ let $Z$ be the complex obtained by gluing $\set{D_{ij}}_{1\leq i<j\leq 4}$ along their boundaries, that is, $ Z=\coprod_{1\leq i<j\leq 4}D_{ij}/\sim$ where $\sim$ is generated by the given bijections.

    \item $f:\coprod_{1\leq i<j\leq 4} D_{ij}\to X$ is a \emph{filling map} of $x_1,x_2,x_3,x_4$ through $o$ where $\set{D_{ij}}_{1\leq i<j\leq 4}$ is a deflated tetrahedron, $x_1,x_2,x_3,x_4,o\in X$, if
    \begin{itemize}
        \item $f$ is a cellular map, up to a refinement of the cellular structure on $X$, and injective on every cell,
        \item $f(y_i)=x_i$,
        \item $f|_{D_{ij}}\ii(\set{o})=\set{\omega}$,
        \item $f|_{\overline{y_iy_j}}$ from $\overline{y_iy_j}$ to the geodesic segment $\overline{x_ix_j}$ in $X$ is a homeomorphism and descends to a map from $Z$.
    \end{itemize}
    
    \item A pair $(\set{D_{ij}}_{1\leq i<j\leq 4},f)$ of a deflated tetrahedron and a filling map of $x_1,x_2,x_3,x_4$ through $o$ will be referred to as a \emph{deflated tetrahedron filling $x_1,x_2,x_3,x_4$ through $o$}.
\end{itemize}
\end{definition}

We note that we do not require $f|_{D_{ij}}$ to be injective, but  we will see that,  under an appropriate minimality assumption (\cref{def: minimal deflated}), $f$ is injective on $D_{ij}$ and so is its induced map on $Z$ (\cref{deflated tetrahedron filling exists}), and that $Z$ is a deformation retract of what one will imagine a deflated tetrahedron to be - a shape like in \cref{fig:def tetra}.

\begin{proposition}\label{deflated tetrahedron filling exists}
For any $x_1,x_2,x_3,x_4\in X$ and $o\in \bigcap_{1\le i<j<k\le 4} \bt{x_i,x_j,x_k}$ there exists a deflated tetrahedron filling $x_1,x_2,x_3,x_4$ through $o$.
\end{proposition}

\begin{proof}
By \cref{triangles are discs}, $\set{\bt{x_i,x_j,o}}_{1\leq i<j\leq 4}$ are spiky triangles.
As in the proof of \cref{triangles are discs}, by \cite{zeeman_1964} for each $(i,j)$ we may refine the cellular structure of $X$ and get finite $\Delta$-complexes $D_{ij}\cong \bt{x_i,x_j,o}$ and cellular, injective maps $f_{ij}:D_{ij}\to X$.

Define $f:\coprod _{1\leq i<j\leq 4}D_{ij}\to X$ to be the function induced from $\set{f_{ij}}_{1\le i <j\le 4}$. There is a refinement of the triangulation of $X$ and of $D_{ij}$ such that $f$ is cellular.

The spiky triangles $D_{ij}$ are in particular generalized discs with three marked points \[y_i:=f_{ij}^{-1}(x_i),y_j:=f_{ij}^{-1}(x_j),\omega:=f_{ij}^{-1}(o).\] 
We claim $(\set{D_{ij}}_{1\leq i<j\leq 4},f)$ is a deflated tetrahedron filling $x_1,x_2,x_3,x_4$ through $o$.
To establish that $\set{D_{ij}}_{1\leq i<j\leq 4}$ is a deflated tetrahedron we need to show that $\set{T_{ijk}}_{1\leq i<j<k\leq 4}$ are spiky triangles, the other parts of the definition are clear from the construction.
Define $a_{ij}$ to be the point in $D_{ij}$ in which $\overline{\omega y_i},\overline{\omega y_j}$ diverge, as in \cref{fig:gen triangle}.

\begin{figure}[H]
    \centering
    \includegraphics{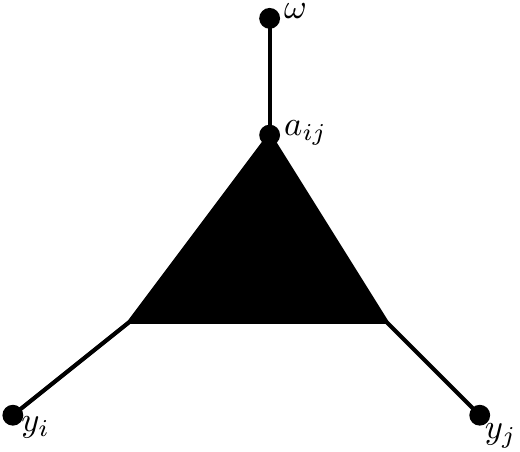}
    \caption{$\bt{y_i,y_j,o}$}
    \label{fig:gen triangle}
\end{figure}

We will argue without loss of generality for $T_{123}$. If $a_{12}=a_{13}=a_{23}=\omega$ then clearly $\set{T_{ijk}}_{1\leq i<j<k\leq 4}$ are spiky triangles, as we are gluing three triangles with no spike at $\omega$.

Suppose instead $a_{12}\neq \omega$. 

Denote by $\hat{x}_i$ the projection of $x_i$ to $\lk(o)$. As $a_{12}\neq \omega$ we have that $\hat{x}_1=\hat{x}_2$, hence $\triangle(\hat{x}_1,\hat{x}_2,\hat{x}_3)=\overline{\hat{x}_1\hat{x}_3}=\overline{\hat{x}_2\hat{x}_3}$. As $o\in \bt{x_1,x_2,x_3}$ we have by \cref{homotopies in link} that $o\in \overline{x_1x_3},\overline{x_2x_3}$.
In this case $\bt{x_1,x_3,o}=\overline{x_1x_3}$ and $\bt{x_2,x_3,o}=\overline{x_2x_3}$, so the gluing of $T_{123}\cong \bt{x_2,x_3,o}$ which is of course a spiky triangle.

We have that $f$ is a filling map of $x_1,x_2,x_3,x_4$ through $o$.
It is piecewise linear and injective on every 2-cell, as it is injective on $D_{ij}$, also from that follows that $f_{ij}^{-1}(\set{(o)})=\set{\omega}$ and by construction $f$ is a homeomorphism of the relevant boundary part to $\overline{x_ix_j}$.
\end{proof}

\begin{definition}\label{def: minimal deflated}
 A deflated tetrahedron $(\set{D_{ij}}_{1\leq i<j\leq 4},f)$ filling $x_1,x_2,x_3,x_4$ through some $o\in \bigcap_{1\le i<j<k\le 4} \bt{x_i,x_j,x_k}$ with a minimal number of cells will be referred to as a \textit{minimal tetrahedron filling $x_1,x_2,x_3,x_4$}.
\end{definition}

\begin{proposition} \label{local injectivity for triangles}
Consider $(\set{D_{ij}}_{1\leq i<j\leq 4},f)$ a minimal tetrahedron filling $x_1,x_2,x_3,x_4$, then for all $1\leq i<j<k\leq 4$ the induced map $f:T_{ijk}\to X$ is near-immersion and hence from \cref{local link for spiky} it is injective. In particular $\set{T_{ijk}}_{1\leq I<j<k\leq 4}$ embed in $Z$.
\end{proposition}

\begin{proof}
Assume towards contradiction that $f$ is not near-immersion. Then, as $f$ is injective on 2-cells, there are twin cells in $T_{ijk}$, i.e. 2-cells $t,t'$ that share a 1-cell and such that $f(t)=f(t')$.

If $t,t'\subseteq D_{ij}$ for some $i,j$ we will commence as in the proof of \cref{triangles are discs}, that is, define $C_{ij}$ the connected component  of $\omega$ in $D_{ij}-\inner(t\cup t')/\sim$, where $\sim$ is the cellular equivalence relation defined by $u\sim v$ if $u$ a 1-cell in $\partial t$  and $v$ a 1-cell in $\partial t'$ and $f(u)=f(v)$, as in \cref{triangles are discs}. Define $C_{jk}=D_{jk}$ for $(j,k)\neq (i,j)$. We get $\set{C_{ij}}_{1\leq i<j\leq 4}$ is a new deflated tetrahedron with fewer cells, as all is still satisfied ($T_{123}$ is still a spiky triangle as explained in the proof of \cref{triangles are discs}), and the induced map from $f$ is a filling map, so we get a contradiction.

Hence $t,t'$ must be in different cell complexes. Without loss of generality assume that $t\subseteq D_{12},t'\subseteq D_{13}$.

We would like to ``fold'' $t,t'$ and insert the folded copy to $D_{14}$, as in \cref{fig:folding}.

\begin{figure}[H]
    \centering
    \includegraphics[]{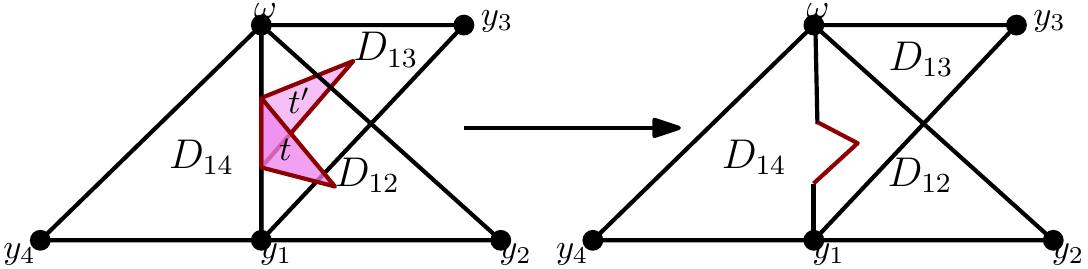}
    \caption{Folding of $t,t'$}
    \label{fig:folding}
\end{figure}

We divide into cases depending on $t\cap t'$:

\textbf{Case 1. The intersection $t\cap t'$ is an interval (one or two 1-cells).} \label{case 1}

Define $C_{12}$ to be a complex obtained from $D_{12}$ removing $\inner(t)$ and the open interval of $t\cap t'$. Similarly, define $C_{13}$ to be the complex  obtained from $D_{12}$ removing $\inner(t')$ and the open interval of $t\cap t'$. Set $C_{14}=D_{14}\cup t$ and $C_{ij}=D_{ij}$ for $2\leq i<j$.
For notation purposes $C_{ij}=C_{ji}$.

Denote by $f':\coprod_{1\leq i<j\leq 4} C_{ij}\to X$ the induced map from $f:\coprod_{1\leq i<j\leq 4} D_{ij}\to X$.
We claim that $(\set{C_{ij}}_{1\leq i<j\leq 4},f)$ is a deflated tetrahedron filling $x_1,x_2,x_3,x_4$ through $o$ and that will give us contradiction to minimality of $(\set{D_{ij}}_{1\leq i<j\leq 4},f)$.

We have that $C_{12}$ and $C_{13}$ are deformation retracts of $D_{12}$ and $D_{13}$ respectively and $D_{14}$ is a deformation retract of $C_{14}$, hence they are planar, finite, contractible $\Delta$-complexes.
Choose $y_i,y_j,\omega$ to be the same marked points as in $D_{ij}$ and get that $\set{C_{ij}}_{1\leq i<j\leq 4}$ are generalized triangles.

The induced map between $\overline{y_io}^{C_{ij}}$ to $\overline{y_io}^{C_{ik}}$ is a bijection.

Denote by $T'_{ijk}$ the gluing induced by the bijections on the boundary parts of $C_{ij},C_{jk},C_{ik}$. $T'_{ijk}$ are spiky triangles for $1\leq i<j<k\leq 4$:
indeed, $T'_{124}=T_{124}$ by working out the definitions of $C_{ij}$. We have $T'_{234}=T_{234}$ as $C_{ij}=D_{ij}$ for the relevant cell complexes. As in the proof of \cref{triangles are discs}, $T'_{123} = T_{123}-(\inner(t\cup t')/\sim$ where $\sim$ identifies the boundary of $t\cup t'$ to an interval, so $T'_{123}$ is a spiky triangle. Lastly, $T'_{134}=(T_{134}-t') \cup t/\sim$ is a spiky triangle, as topologically we deleted a disc and glued another disc instead.

Clearly, $f'$ is a filling map. This gives a contradiction to minimality.

\textbf{Case 2. The intersection $t\cap t'$ is a disjoint union of a 1-cell and 0-cell $v$.} 

There are two further cases, either $v\in \overline{y_1\omega}$, as in \cref{fig:first}, or $v\in\overline{y_2\omega}$, as in \cref{fig:second} (where in this case it happens to be that $D_{23}$ has a skinny part).

\textbf{Case 2.1. $v\in \overline{y_1\omega}$.} 

\begin{figure}[htp]
    \centering
    \includegraphics[]{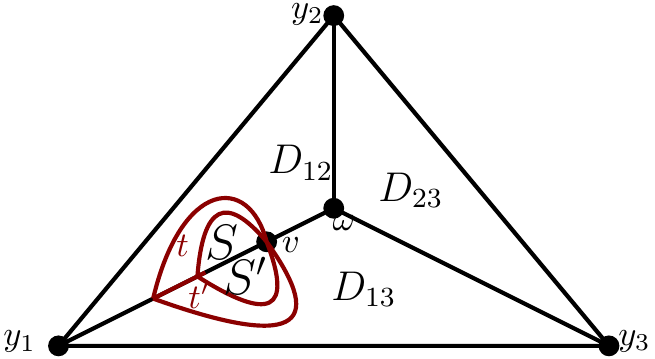}
    \caption{Case 2.1.}
    \label{fig:first}
\end{figure}
We have that $D_{12}-t$ has two connected components. Denote by $C_{12}$ the closure of the connected component containing $\omega$ and by $S$ the closure of the other connected component and let $T=S\cup t$.
Define $C_{13},S'$ similarly to be the connected components of $D_{13}-t'$ containing $\omega$ and not containing $\omega$ respectively and let $T'=S'\cup t'$. Set $C_{14}= D_{14}\cup T$ and $C_{ij}=D_{ij}$ for $2\leq i<j$.   
Let $f'$ be the induced map. 

We argue identically to case 1, with $T,T'$ instead of $t,t'$, to show that $(\set{C_{ij}}_{1\leq i<j\leq 4},f')$ is a deflated tetrahedron filling $x_1,x_2,x_3,x_4$ through $o$.
This is again a contradiction to minimality.

\textbf{Case 2.2. $v\in\overline{y_2\omega}$.} 

\begin{figure}[htp]
    \centering
    \includegraphics[]{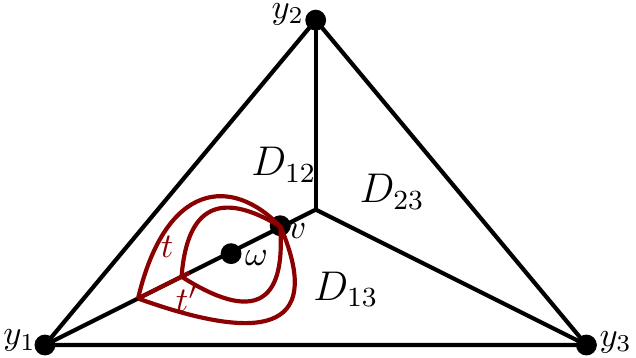}
    \caption{Case 2.2.}
    \label{fig:second}
\end{figure}

Let $A$ be the connected component of $y_1$ in $T_{123}-(t\cup t')$. It is a generalized annulus with spikes and $\omega\notin A$, as in \cref{triangles are discs}. Consider $\sim$ the cellular equivalence relation defined by $u\sim v$ if $u$ a 1-cell in $\partial t$  and $v$ a 1-cell in $\partial t'$ and $f(u)=f(v)$, as in \cref{triangles are discs}. It is a good gluing on $A$ and so from \cref{gluing an annulus}, $A/\sim$ is a spiky triangle, in particular, it is simply connected. The map $\bar{f}$ induced from $f$ maps the boundary of $A$ homeomorphically onto $\triangle(x_1,x_2,x_3)$, hence $\bt{x_1,x_2,x_3}\subseteq \bar{f}(T)$ which is a contradiction to $f$ being a filling map and satisfying $f|_{C_{ij}}^{-1}(\set{o})=\set{\omega}$ and the fact that $\omega\notin A$ while $f(\omega)=o\in \bt{x_1,x_2,x_3}$.
\end{proof}

\begin{corollary} \label{local injectivity for min tetra}
For $(\set{D_{ij}}_{1\leq i<j\leq 4},f)$ a minimal tetrahedron filling $x_1,x_2,x_3,$ $x_4$, the map $f:Z\to X$ is a near-immersion.
\end{corollary}

\begin{proof}
$f$ is injective on cells by assumption. Assume towards contradiction that there are $t,t' \subseteq Z$ twin cells. As $f|_{T_{ijk}}$ is injective, by \cref{local injectivity for triangles} we have that $t\subseteq D_{ij},t'\subseteq D_{kl}$ for $\set{i,j,k,l}=\set{1,2,3,4}$. From this, we have $f(t)=f(t')\subseteq \bigcap_{1\leq i<j<k\leq 4}\bt{x_i,x_j,x_k}$, but $f(t)$ is a 2-cell in $X$ and this is a contradiction to the fact that the intersection is a 1-dimensional complex, by
\cref{discreteness}.
\end{proof}

\begin{corollary} \label{injectivity on D}
Assume $x_1,x_2,x_3,x_4\in X$ do not form a lanky quadrilateral. If $(\set{D_{ij}}_{1\leq i<j\leq 4},f)$ is a minimal tetrahedron filling $x_1,x_2,x_3,x_4$, then each $f|_{D_{ij}}:D_{ij}\to X$ is injective.
\end{corollary}

\begin{proof}
Without loss of generality $(i,j)=(1,2)$.
As $f|_{T_{123}}$ is injective we have that if $f|_{D_{12}}$ is not injective then it must be that there are two points on the boundary which are identified, that is, $D_{23},D_{13}$ have a skinny part, as in \cref{fig:non injectivity}. 

\begin{figure}[H]
    \centering
    \includegraphics[]{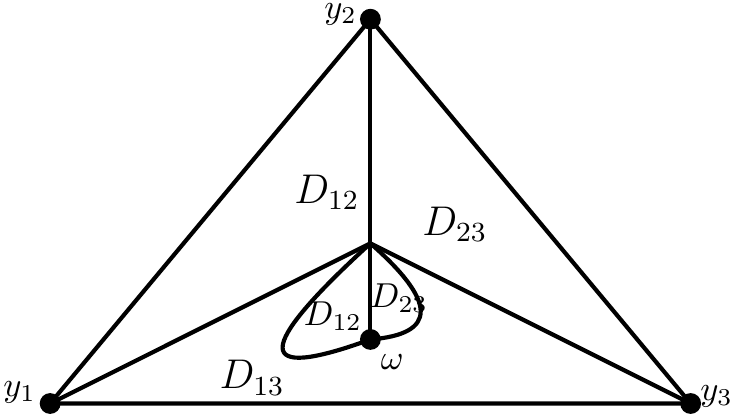}
    \caption{non injectivity of $D_{12}$}
    \label{fig:non injectivity}
\end{figure}

There are two cases, either $e_1,e_2$, the two 1-cells adjacent to $\omega$ in the boundary of $D_{12}$, are identified in $T_{123}$, or not. In the latter case, we consider the pair of points in $D_{12}$ which are identified and are closest to $\omega$ and get that $D_{12}$ bounds a disc in $T_{123}$.

\textbf{Case 1. $e_1,e_2$ are identified in $T_{123}$. }

Assume $e_1\subseteq \overline{y_1\omega}^{12},e_2\subseteq \overline{y_2\omega}^{12}$. Note that $e_1\subseteq T_{124}, T_{123},T_{134}$ and $e_2\subseteq T_{234}$ so $f(e_1)=f(e_2)\subseteq \bigcap_{1\leq i<j<k\leq 4} f(T_{ijk})= \bigcap_{1\leq i<j<k\leq 4} \bt{x_i,x_j,x_k}$. By \cref{discreteness}, we have that $f(e_1)$ is discrete, which is a contradiction to $f|_{e_1}$ is injective.

\textbf{Case 2. $D_{12}$ bounds a disc in $T_{123}$. }

The disc $D_{12}$ bounds in $T_{123}$ is a subset of $D_{13}\cup D_{23}$.
As $f|_{T_{124}}$ is injective we must have that under gluing the boundaries the same points are identified in $T_{124}$ which gives us a disc in $D_{14}\cup D_{24}\subseteq T_{124}$. In $Z$ these two discs glue to a sphere. By \cref{local injectivity for min tetra}, $f$ restricted to that sphere is a near-immersion which is a contradiction, as we can pull back the CAT(0) structure from $X$ to the sphere.
\end{proof}

\begin{proposition}\label{minimal deflated is injective}
Assume $x_1,x_2,x_3,x_4\in X$ do not form a lanky quadrilateral and $(\set{D_{ij}}_{1\leq i<j\leq 4},f)$ a minimal deflated tetrahedron filling $x_1,x_2,x_3,x_4$, then $f:Z\to X$ is injective.
\end{proposition}

\begin{proof}
By \cref{local injectivity for min tetra} and \cref{injectivity on D}, $f|_Z$ is a near-immersion and $f|_{D_{ij}}$ is injective.

Assume towards contradiction $p,q\in Z$ distinct and $f(p)=f(q)$. Without loss of generality $p\in D_{12}$.
If $q\in \bigcup_{(i,j)\neq (3,4)}D_{ij}$ there exists $i<j<k$ such that $p,q\in T_{ijk}$ are distinct. That is a contradiction by \cref{local injectivity for triangles} that states that $f|_{T_{ijk}}$ is injective.

So it must be that $q\in \inner(D_{34})\cup \overline{y_3y_4}$ and similarly $p\in \inner(D_{12})\cup \overline{y_1y_2}$.
We get $f(p)=f(q)\in \bt{x_1,x_2,x_4}\cap \bt{x_1,x_3,x_4}$, so by \cref{intersection of two triangles} we have that the connected component of $f(p)$ in $\bt{x_1,x_2,x_4}\cap \bt{x_1,x_3,x_4}$ is more than a point, i.e. there is a 1-cell or 2-cell $s\subseteq D_{12}$ containing $p$ such that $f(s)\subseteq \bt{x_1,x_3,x_4}$. But from the injectivity of $f|_{T_{123}},f|_{T_{124}}$ we get that $f(s)\subseteq f(D_{34})$, hence $f(s)\subseteq \bigcap_{1\leq i<j<k\leq 4} \bt{x_i,x_j,x_k}$. This is a contradiction to \cref{discreteness} that states that the intersection is discrete, as $f|_s$ is injective. Hence $f:Z\to X$ is injective.
\end{proof}

Now we can complete the proof of \cref{main theorem} in the case where $x_1,x_2,x_3,x_4\in X$ do not form a lanky quadrilateral,:

\begin{theorem*}
Assume $x_1,x_2,x_3,x_4\in X$ do not form a lanky quadrilateral, then the intersection of $\bt{x_i,x_j,x_k}$ for ${1\leq i<j<k\leq 4}$ is a single point.
\end{theorem*}

\begin{proof}
We will show that in a minimal deflated tetrahedron filling $x_1,x_2,x_3,x_4$, we have $\bigcap_{1\leq i<j<k\leq 4} T_{ijk}=\set{\omega}$. As $f:Z\to X$ is injective, $f(T_{ijk})=\bt{x_i,x_j,x_k}$, which will show that $\bigcap_{1\leq i<j<k\leq 4}\bt{x_i,x_j,x_k}=\set{f(\omega)}=\set{o}$. 

Assume $q\in \bigcap_{1\leq i<j<k\leq 4} T_{ijk}$. Without loss of generality $q\in D_{12}$. We get $q\in D_{12}\cap T_{134}= \overline{y_1\omega},q\in D_{12}\cap T_{234}= \overline{y_2\omega}$.
It suffices to show that $\omega$ is the only point that satisfies this, i.e. that $D_{ij}$ are actually discs with two spikes where $y_i,y_j$ are the endpoints of each spike.

Assume towards contradiction $q\neq \omega$.

Note $f(q)\in \bt{x_1,x_3,x_4}\cap\bt{x_2,x_3,x_4}$.
By \cref{intersection of two triangles}, there is a 1-cell $e$ in $Z$ containing $q$ such that $f(e)\subseteq \bt{x_1,x_3,x_4}\cap\bt{x_2,x_3,x_4}$.

From injectivity of $f$ on $Z$, the 1-cell $e$ must be in $D_{34}$ or in  $\overline{y_1\omega}\cap \overline{y_2\omega}$. 

\textbf{Case 1. $e$ in $D_{34}$.}

In this case $q\in D_{34}$, that is $q\in \overline{y_i\omega}$ for $i=3$ or $4$. Without loss of generality $q\in \overline{y_3\omega}$.
For $i,j\in \set{1,2,3}$ different $f|_{D_{ij}}$ is injective by \cref{injectivity on D} hence there exists a unique point $q_{ij}\in D_{ij}$ such that its image in $T_{123}$ is $q$.

    Note $q_{ij}$ separates $D_{ij}$ for every $1\leq i<j\leq 4$. Consider $C_{ij}$ the connected component of $y_i$ in $D_{ij}-\set{q_{ij}}$. Their image in $T_{123}$ is simply connected by Van Kampen. This is a contradiction to $o\in \bt{x_1,x_2,x_3}$.

\textbf{Case 2. $e$ in  $\overline{y_1\omega}\cap \overline{y_2\omega}$.}

By choice of $e$, $f(e)\subseteq \bt{x_1,x_3,x_4}\cap \bt{x_2,x_3,x_4}$ and by the assumption in this case, $f(e)\subseteq \bt{x_1,x_2,x_3}\cap\bt{x_1,x_2,x_4}$, so $f(e)\subseteq \bigcap_{1\le i<j<k\le 4} \bt{x_i,x_j,x_k}$. This is a contradiction to the fact that the intersection is discrete, stated in \cref{discreteness} and the fact that $f|_e$ is injective.

\end{proof}

\section{Future directions}\label{sec:conclusion}

\paragraph{Higher dimensions.} 
It is natural to try and generalize this work to higher dimensions. 
To do this, one needs to generalize triangles and $n$-median spaces. We suggest the following inductive definition for the generalisation of triangles.

\begin{definition*}
Let $\blacktriangle^2(x_1,x_2,x_3)=\bt{x_1,x_2,x_3}$. Let $x_1,\dots ,x_{n+1}\in X$ where $X$ is a CAT(0) complex and define $\blacktriangle^n(x_1,\dots ,x_{n+1})$ to be the minimal ball filling $\bigcup_{i=1}^{n+1}\blacktriangle^{n-1}(x_1,\dots ,\hat{x}_i,\dots ,x_{n+1})$.
\end{definition*}
One can study intersections of $\blacktriangle^n(x_1,\dots ,x_{n+1})$ for all $x_1,\dots,x_{n+1}$ in a set of $n+2$ points. 



\paragraph{Median algebras.} We recall that for a set $M$ and a map $\mu:M^3\to M$ the pair $(M,\mu)$ is a \emph{median algebra} if for any $a,b,c,d,e\in M$ we have $\mu(a,b,c)=\mu(b,c,a)=\mu(b,a,c)$, $\mu(a,a,b)=a$ and $\mu(a,b,\mu(c,d,e))=\mu(\mu(a,b,c),\mu(a,b,d),e)$. A median space with the median map is clearly a median algebra. A natural question to ask is if there is a meaningful definition of ``2-median algebras'', such that 2-median spaces will satisfy this definition, and characterize them.

\paragraph{Coarse median.} 
A geodesic metric space $(x,d)$ is \emph{coarse median} if 

\begin{itemize}
    \item there is a map $\mu:X^3\to X$ and $k,l>0$ such that for any $a,b,c,a',b',c'\in X$ we have $d(\mu(a,b,c),\mu(a',b',c')<k(d(a,a')+d(b,b')+d(c,c'))+l$,
    \item there is a map $h:\mathbb{N}\to [0,\infty)$ such that if $A\subseteq X$ such that $\abs{A}\leq p$ then there exists a finite median algebra $(M,m)$ and maps $\pi:A\to M,\lambda:M\to A$ such that for all $a\in A$ $d(a,\lambda(\pi(a)))<h(p)$ and for all $a,b,c\in M$ $d(\mu(\lambda(a),\lambda(b),\lambda(c)),\lambda(m(a,b,c)))<h(p)$.
\end{itemize}  

It is natural to try to define coarse 2-median spaces and study which spaces are coarse 2-median. In particular, are rank-2 symmetric spaces coarse 2-median?

\appendix

\section{Appendix -- Fary Milnor for CAT(0) polygonal complexes}\label{sec: appendix}
In this appendix we will recover a special case of a result of  Stadler \cite{stadler2021}, proving a Fary-Milnor type theorem for CAT(0) polygonal complexes. Moreover, in this case, the image of the embedded disc is unique.

Recall that for a polygonal curve $\Gamma$ we define its \emph{total curvature} to be $$\kappa(\Gamma)=\sum_{i=1}^k(\pi-\angle_{x_i}\Gamma),$$ where $x_1,\ldots,x_k$ are the vertices of $\Gamma$.

\begin{theorem} [Fary-Milnor for CAT(0) polygonal complexes] \label{Fary-Milnor}
Let $X$ be a CAT(0) polygonal complex, and let $\Gamma$ be an embedded polygonal curve. If $\kappa(\Gamma)<4\pi$ then $\Gamma$ bounds an embedded disc in $X$.

The image of the disc is unique, by \cref{injective discs}.
\end{theorem}

We note that $\kappa(\triangle(x,y,z))<4\pi$.

\cref{Fary-Milnor} follows from \cref{local link}, if we show that there exists a near-immersion $f:\mathbb{D}^2\to X$ filling $\Gamma$. This is similar to Lemma 3.3 in \cite{Przytycki}.

\begin{proof}
By subdividing the cells of $X$, we may assume $X$ is a triangle complex, i.e. it has a simplicial structure.

As $X$ is simply connected, there exists a map $f:\bbD^2\to X$ such that  $f|_{\partial \bbD^2}$ maps ${\partial \bbD^2}$ homeomorphically onto $\Gamma$.
The homeomorphism $f|_{\partial \bbD^2}: \partial \bbD^2 \to \Gamma$ can be made cellular by subdividing $\partial \bbD^2$ and $X$. By the simplicial approximation theorem (see e.g. \cite{zeeman_1964}), there exists $T$,  a triangulation of $\bbD^2$, and a map $f':T \to X$ that is cellular with respect to some simplicial structure of $X$ and that maps $\partial T$ homeomorphically onto $\Gamma$.

Fix the simplicial structure on $X$ from above and consider the collection of all pairs $(h,D)$ where $D$ is a disc with some finite $\Delta$-complex structure and $h:D\to X$ is a cellular map, such that $h|_{\partial D}:\partial D \to \Gamma$ is a homeomorphism.
Consider $(f,T)$ in this collection such that the number of cells in  $T$ is minimal among all pairs in the collection. We claim $f$ is a near-immersion. 

To simplify future arguments we show first that $T$ has a simple structure - all 1-cells in $T^{(1)}$ are intervals (not loops) and there are no double 1-cells.

\begin{claim} \label{no loops}
The 1-skeleton of $T$ is a simple graph.
\end{claim}
\begin{proof}
\textbf{$T^{(1)}$ contains no loops.}
As $X$ is a simplicial complex, if $e$ is a 1-cell in $T^{(1)}$ then $f(e)$ is either a point or a segment in $X^{(1)}$, so if $e$ is a loop in $T$ then $f(e)$ is a point. Such $e$ must be in the interior of $T$, as the boundary of $T$ is mapped injectively. 

Consider $e$ an outermost such loop in $T$. We have that $e$ separates $T$ into a generalized annulus and a disc, as in \cref{fig:loop}, and it is on the boundary of two 2-cells $t,t'$, where $t$ is in the generalized annulus. Consider the closure $A$ of the connected component of $\partial T$ in $T-t$. Let $e_1,e_2$ be the two 1-cells on the boundary of $t$ other than $e$ oriented such that $t(e_1)=o(e_2)$. Since $e$ is an outermost loop, $e_1,e_2$ are not loops, hence their concatenation forms a simple loop which is the inner boundary of $A$. 

As $f$ is cellular and not injective on $t$ we have $f(t)=f(\partial t)$, which is a 0-cell or 1-cell in $X$, so $f(e_1)=f(e_2)$. Let $\sim$ be the cellular equivalence relation defined by $e_1\sim e_2^{-1}$. The relation $\sim$ is a good gluing, as it identifies the two 1-cells on the inner boundary of $A$ in the correct orientation, hence by \cref{gluing an annulus}, $A/\sim$ is a disc.

The induced $\Delta$-complex structure on the disc $A/\sim$ has fewer cells than the number of cells in $T$ and the map induced from $f$ is still a filling map. This is a contradiction to the minimality of $(f,T)$.

\begin{figure}[H]
    \centering
    \includegraphics[]{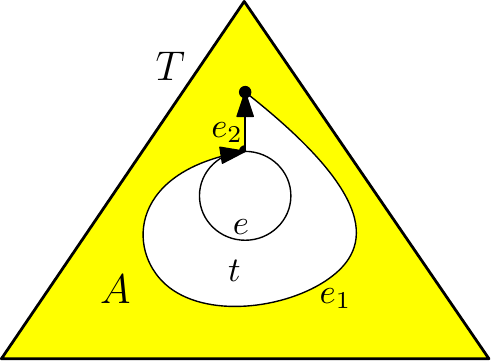}
    \caption{$e$ is a loop in $T$}
    \label{fig:loop}
\end{figure}

\textbf{$T^{(1)}$ contains no double 1-cells.}
Assume $e_1,e_2$ are oriented 1-cells with $o(e_1)=o(e_2),t(e_1)=t(e_2)$, then as $X$ is a simplicial complex $f(e_1)=f(e_2)$. 
The concatenation of $e_1$ and $e_2$ separates $T$ into a generalized annulus $A$ and a disc, as in \cref{fig:double edges}. The inner boundary of $A$ is the cycle $e_1,e_2 ^{-1}$. Let $\sim$ be the cellular equivalence relation $e_1\sim e_2$. As in the loop case, $\sim$ is a good gluing, as it identifies the two 1-cells on the inner boundary of $A$ in the correct orientation, hence by \cref{gluing an annulus}, $A/\sim$ is a disc.

The induced $\Delta$-complex structure on the disc $A/\sim$ has less cells than the number of cells in $T$ and the map induced from $f$ is still a filling map. This is a contradiction to the minimality of $(f,T)$.

\begin{figure}[H]
    \centering
    \includegraphics[]{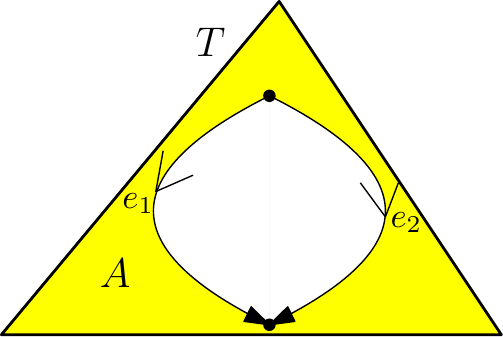}
    \caption{$e_1,e_2$ are double 1-cells}
    \label{fig:double edges}
\end{figure}

\end{proof}

If $f|_T$ is not a near-immersion then either there is a cell $t$ such that $f(t)=f(\partial t)$ or there are 2-cells $t,t'$ which share at least one 1-cell on their boundary and such that $f(t)=f(t')$. We will show both cases cannot happen.

\textbf{Case 1.} Assume $f(t)=f(\partial t)$. In this case, there exists a 1-cell, $e$, in $\partial t$ that is mapped to a point. As $f|_{\partial T}$ is injective, $e$ must be in the interior of $T$. Let $t'$ be the other 2-cell incident to $e$. Let $e,e_1,e_2$ and $e,d_1,d_2$ be the 1-cells in $\partial t$ and $\partial t'$ respectively, oriented as a cycle, as in \cref{fig:1-cell}. 
    
    Let $A$ be the closure of the connected component of $\partial T$ in $T-(\bar{t}\cup \bar{t'})$.
    The subcomplex $A$ is the closure of the complement of a closed disc in $T$, hence it is a generalized annulus. By the claim above, \cref{no loops}, the inner boundary of $A$ is a path consisting of four oriented 1-cells $e_1,e_2,d_2^{-1},d_1^{-1}$ in this order, see \cref{fig:1_cell}.
    
    We have that $f(e_1)$ is the inverse of $f(e_2)$ and $f(d_1)$ is the inverse of $f(d_2)$.
    Let $\sim$ be the cellular equivalence relation defined by $e_1\sim e_2^{-1},d_1\sim d_2^{-1}$. It is a good gluing, as $f$ is injective on the boundary of $T$, from strongly non-collinearity of $x,y,z$. From \cref{gluing an annulus}, $A/\sim$ is a disc which is a finite $\Delta$-complex. 
    The map induced from $f$, denoted $\bar{f}:A/\sim \to X$, is well defined, cellular and $\bar{f}|_{\partial A}:\partial A\to \triangle(x,y,z)$ is a homeomorphism. But $A/\sim$ has fewer cells than $T$, this is a contradiction to the minimality of $(f,T)$.

    \begin{figure}[H]
    \centering
    \includegraphics{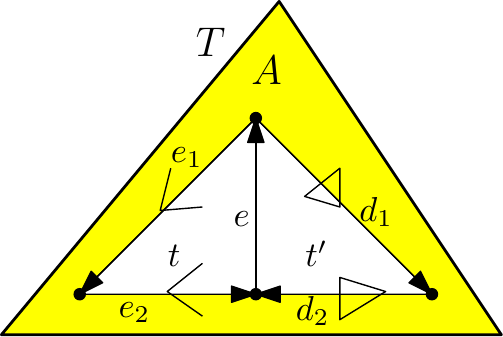}
    \caption{The case $f(t)=f(\partial t)$}
    \label{fig:1_cell}
    \end{figure}

\textbf{Case 2.} Assume $f(t)=f(t')$ for $t,t'$ sharing a 1-cell $e$. Let $e,e_1,e_2$ and $e,d_1,d_2$ be the 1-cells in $\partial t$ and $\partial t'$ respectively, oriented as a cycle, as in \cref{fig:twin_cells}.
    
    Let $A$ be the closure of the connected component of $\partial T$ in $T-(\bar{t}\cup \bar{t'})$. 
    As above, $A$ is a generalized annulus and the inner boundary of $A$ is a path consisting of four oriented 1-cells $e_1,e_2,d_2^{-1},d_1^{-1}$ in this order, as in \cref{fig:twin_cells}. Because $f(t)=f(t')$, we have that $f(e_1)$ is the inverse of $f(d_2)$ and $f(e_2)$ is the inverse of $f(d_1)$. 
    
    Let $\sim$ be the cellular equivalence relation defined by $e_1\sim d_2^{-1},d_1\sim e_2^{-1}$. It is a good gluing, as before, as $x,y,z$ are strongly non-collinear. By \cref{gluing an annulus}, we have that $A/\sim$ is homeomorphic to a disc. It is a finite $\Delta$-complex with fewer cells than $T$, and the induced map $\bar{f}:A/\sim \to X$ is still cellular. This contradicts the minimality of $(f,T)$.

    \begin{figure}[H]
    \centering
    \includegraphics{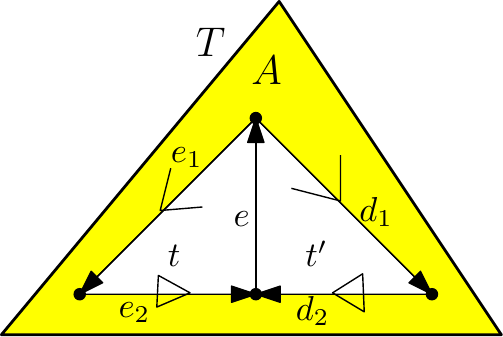}
    \caption{The case $f(t)=f(t')$}
    \label{fig:twin_cells}
    \end{figure}

Hence $f$ is a near-immersion.

By \cref{local link}, we get that $f$ is injective and this concludes the proof.

\end{proof}

\printbibliography


\end{document}

%% file: preamble.tex
\usepackage[leqno]{amsmath}
\usepackage{amssymb,amsfonts,xfrac,MnSymbol}
\usepackage{amsthm}
\usepackage{hyperref}
\usepackage[capitalise]{cleveref}
\usepackage{todonotes}
\usepackage{enumitem}
\usepackage{graphicx}
\usepackage{tikz-cd, tikz}
\usepackage{color}
\usepackage{import}
\usepackage{float}

\numberwithin{equation}{section}

\newtheorem{theorem}{Theorem}[section]
\newtheorem{claim}[theorem]{Claim}
\newtheorem{proposition}[theorem]{Proposition}
\newtheorem{lemma}[theorem]{Lemma}
\newtheorem{corollary}[theorem]{Corollary}

\newtheorem*{theorem*}{Theorem}
\newtheorem*{claim*}{Claim}
\newtheorem*{proposition*}{Proposition}
\newtheorem*{lemma*}{Lemma}
\newtheorem*{corollary*}{Corollary}

\newtheorem{theoremA}{Theorem}

\theoremstyle{definition}
\newtheorem{definition}[theorem]{Definition}

\newtheorem{remark}[theorem]{Remark}
\newtheorem{example}[theorem]{Example}

\newtheorem{notation}[theorem]{Notation}

\newtheorem*{definition*}{Definition}
\newtheorem*{observation*}{Observation}
\newtheorem*{remark*}{Remark}
\newtheorem*{example*}{Example}
\newtheorem*{question*}{Question}
\newtheorem*{exercise*}{Exercise}
\newtheorem*{fact*}{Fact}
\newtheorem*{notation*}{Notation}

\newcommand{\bbD}{\mathbb{D}}
\newcommand{\bbE}{\mathbb{E}}

\newcommand{\bbR}{\mathbb{R}}
\newcommand{\bbS}{\mathbb{S}}

\newcommand{\ii}{^{-1}}

\newcommand{\bt}[1]{\blacktriangle \prs {#1}}

\DeclareMathOperator{\conv}{conv}
\DeclareMathOperator{\inner}{int}

\DeclareMathOperator{\lk}{lk}
\DeclareMathOperator{\girth}{girth}

\newcommand{\set}[1]{\left\{ #1 \right\}}
\newcommand{\prs}[1]{\left( #1 \right)}
\newcommand{\brs}[1]{\left[ #1 \right]}